\documentclass[10pt]{amsart}
\usepackage{latexsym}
\usepackage{amsmath}
\usepackage{amssymb}
\usepackage{mathrsfs}
\usepackage{graphicx}
\usepackage{color}
\usepackage{pgfpages}
\usepackage{ifthen}
\usepackage{leftidx,tensor}
\usepackage[T1]{fontenc}
\usepackage[latin1]{inputenc}
\usepackage{mathtools}
\usepackage{comment}
\usepackage{dsfont}
\usepackage[nocompress]{cite}

\usepackage[shortlabels]{enumitem}
\usepackage{aliascnt}
\usepackage[bookmarks=true,pdfstartview=FitH, pdfborder={0 0 0}, colorlinks=true,citecolor=red, linkcolor=blue]{hyperref}
\usepackage{bbm}
\usepackage{nicefrac}

\theoremstyle{plain}
\newtheorem{thm}{Theorem}[section]
\newaliascnt{cor}{thm}
\newaliascnt{prop}{thm}
\newaliascnt{lem}{thm}
\newtheorem{cor}[cor]{Corollary}
\newtheorem{prop}[prop]{Proposition}
\newtheorem{lem}[lem]{Lemma}
\aliascntresetthe{cor}
\aliascntresetthe{prop}
\aliascntresetthe{lem} 
\theoremstyle{definition}
\newaliascnt{defn}{thm}
\newaliascnt{asu}{thm}
\newaliascnt{con}{thm}

\newtheorem{asu}[asu]{Assumption}

\aliascntresetthe{defn}
\aliascntresetthe{asu}
\aliascntresetthe{con}
\newcounter{stp}
\newcounter{stpi}
\newcounter{stpci}
\newcounter{stpiii}

\newtheorem{step}[stp]{\rm\em Step}

\theoremstyle{thm}
\newaliascnt{rem}{thm}
\newaliascnt{exa}{thm}
\newaliascnt{masu}{thm}
\newaliascnt{nota}{thm}
\newaliascnt{sett}{thm}
\newtheorem{rem}[rem]{Remark}

\aliascntresetthe{rem}
\aliascntresetthe{exa}
\aliascntresetthe{masu}
\aliascntresetthe{nota}
\aliascntresetthe{sett}
\numberwithin{equation}{section}

\setlist[enumerate]{font = \normalfont}

\newcommand {\N}	{\mathbb{N}}

\newcommand {\R}	{\mathbb{R}}

\newcommand {\E}	{\mathbb{E}}
\newcommand {\F}	{\mathbb{F}}
\newcommand {\G}	{\mathbb{G}}
\newcommand {\T}	{\mathbb{T}}

\renewcommand{\d}{\, \mathrm{d}}

\DeclareMathOperator{\mdiv}{div}
\DeclareMathOperator{\divH}{div_{\H}}

\newcommand{\Hinfty}{\mathcal{H}^\infty}

\newcommand{\sD}{\mathcal{D}}

\newcommand{\sL}{\mathcal{L}}

\newcommand{\D}{\mathrm{D}}
\newcommand{\rN}{\mathrm{N}}
\renewcommand{\H}{\mathrm{H}}

\newcommand{\loc}{\mathrm{loc}}

\newcommand{\bc}{\mathrm{b.c.}}

\newcommand{\sigmabar}{\bar{\sigma}}

\newcommand{\atm}{\mathrm{a}}
\newcommand{\ocn}{\mathrm{o}}

\newcommand{\air}{\mathrm{a}}

\newcommand{\ssatm}[1]{{#1}^{\atm}}
\newcommand{\ssocn}[1]{{#1}^{\ocn}}

	\newcommand{\hatm}{\ssatm{h}}

	\newcommand{\Omegaatm}{\ssatm{\Omega}}
	\newcommand{\Omegaocn}{\ssocn{\Omega}}
	
	\newcommand{\GaD}{\Gamma_{\D}}
	
	\newcommand{\Gau}{\Gamma_u^{\ocn}}
	\newcommand{\Gab}{\Gamma_b^{\ocn}}

	\newcommand{\dk}[1]{\partial_{#1}}
	\newcommand{\dt}{\dk{t}} 
	\newcommand{\dz}{\dk{z}} 
	
	\newcommand{\tr}{\mathrm{tr}}

	\newcommand{\eps}{\varepsilon}
	\renewcommand{\phi}{\varphi}
	\newcommand{\vatm}{\ssatm{v}}
	\newcommand{\vocn}{v^{\ocn}}
	\newcommand{\wocn}{w^{\ocn}}

	\newcommand{\uatm}{u^{\atm}}
	\newcommand{\uocn}{u^{\ocn}}

	\newcommand{\fatm}{\ssatm{f}}
	\newcommand{\focn}{f^{\ocn}}

	\newcommand{\xH}{x_\mathrm{H}}
	
	\renewcommand{\bar}[1]{\overline{#1}}
	\newcommand{\vbar}{\bar{v}}
	\newcommand{\fbar}{\bar{f}}
	\newcommand{\vtilde}{\tilde{v}}
	\newcommand{\ftilde}{\tilde{f}}
	\newcommand{\vbaratm}{\ssatm{\vbar}}
	\newcommand{\vbarocn}{\ssocn{\vbar}}

	\newcommand{\nablaH}{\nabla_{\H}}
	\newcommand{\DeltaH}{\Delta_{\H}}

	\newcommand{\hra}{\hookrightarrow}
	
	\newcommand{\rC}{\mathrm{C}}
	\newcommand{\rL}{\mathrm{L}}
	\newcommand{\rW}{\mathrm{W}}
	\newcommand{\rH}{\H}
	\newcommand{\rB}{\mathrm{B}}
	
	\newcommand{\rF}{\mathrm{F}}

	\newcommand{\rS}{\mathrm{S}}

	\newcommand{\rLq}{\rL^q}
	\newcommand{\rLp}{\rL^p}

	\newcommand{\on}[1]{\text{ on } {#1}}
	\renewcommand{\for}{\text{ for }}

	\newcommand{\onOmega}{\on{\Omega}}
	
	\newcommand{\onOmegaT}{\on{\Omega \times (0,T)}}

	\newcommand{\uair}{u^{\air}}
	\newcommand{\vair}{v^{\air}}

	\newcommand{\omegaair}{\omega^{\air}}

	\newcommand{\pair}{p^{\air}}

	\newcommand{\fair}{f^{\air}}
	
	\newcommand{\vairbar}{\bar{v}^{\air}}
	\newcommand{\vocnbar}{\bar{v}^{\ocn}}
	\newcommand{\vairtilde}{\tilde{v}^{\air}}
	\newcommand{\vocntilde}{\tilde{v}^{\ocn}}
	
	\newcommand{\Omegaair}{\Omega^{\air}}
	\newcommand{\rX}{\mathrm X}
    \newcommand{\Gammaairu}{\Gamma_u^{\atm}}
    \newcommand{\Gammaairb}{\Gamma_b^{\atm}}
    
    \newcommand{\wair}{w^\atm}

\title[Global Strong Well-Posedness of the CAO-Problem introduced by Lions, Temam and Wang]{Global strong Well-Posedness of the CAO-Problem introduced by Lions, Temam and Wang}
\author{Tim Binz\textsuperscript{1}}
\address{\textsuperscript{1}Princeton University\\
       Program in Applied \& Computational Mathematics\\
       Fine Hall, Washington Road\\
       08544 Princeton, NJ\\
       USA}
\email{tb7523@princeton.edu}

\author{Felix Brandt\textsuperscript{2}}\label{A2}
\address{\textsuperscript{2}University of California, Berkeley\\
	Department of Mathematics\\
	Evans Hall, University Dr\\
	94720 Berkeley, CA\\
	USA}
\email{fbrandt@berkeley.edu}

\author{Matthias Hieber\textsuperscript{3}}\label{A3}
\address{\textsuperscript{3}Technische Universit\"at Darmstadt\\
	Fachbereich Mathematik\\
	Schlossgartenstr.\ 7\\
	64289 Darmstadt\\
	Germany}
\email{hieber@mathematik.tu-darmstadt.de}

\author{Tarek Z\"{o}chling\textsuperscript{3}}\label{A4}
\email{zoechling@mathematik.tu-darmstadt.de}


\begin{document}

\maketitle	

\begin{center}
    \textsuperscript{1}Program in Applied \& Computational Mathematics, Princeton University

    \textsuperscript{2}Department of Mathematics, University of California, Berkeley

    \textsuperscript{3}Fachbereich Mathematik, Technische Universit\"{a}t Darmstadt 
\end{center}
 
\begin{abstract}
Consider the CAO-problem introduced by Lions, Temam and Wang, which concerns  a system of two fluids described by two primitive equations coupled by fully nonlinear interface conditions. 
They proved in their pioneering work the existence of a weak solution to the CAO-system; its uniqueness remained  an open problem. In this article, it is shown  that this coupled CAO-system 
is globally strongly well-posed for large data, even in critical Besov spaces. It is furthermore shown that, away from the boundary, the solution is even real analytic. The approach presented relies on an 
optimal data result for the boundary terms in the linearized system in terms of time-space Triebel-Lizorkin spaces. Boundary terms are then controlled by paraproduct methods in these spaces.       
\end{abstract}

\section{Introduction}\label{sec:intro}
Geophysical fluid dynamics constitutes a thriving scientific field, and its flows are characterized by a multitude of characteristic length and time scales and by their interactions with a 
range of dynamically relevant processes. In this context the primitive equations serve as a standard model for oceanic and atmospheric dynamics and are derived from the 
Navier-Stokes equations by assuming a hydrostatic balance for the pressure term.   
They have been introduced in a series of papers  by Lions, Temam and Wang ~\cite{LTW:92a, LTW:92b, LTW:93, LTW:95} about 30 years ago, see also \cite{PTZ:09}. The primitive equations are, however, only one 
building block of a coupled  system introduced  in the articles cited above and describing  a two-phase system of fluids subject to fully  nonlinear boundary conditions. This system, denoted by 
them by CAO, was derived carefully by them investigating boundary layer effects near the interface.  The rigorous analysis of the CAO-system has a very  rich history but, on the other hand, 
many questions are left open until today. For general  information on geophysical flows we refer e.\ g.\ to the work of Chemin, Desjardins, Gallagher and Grenier \cite{CDGG:06} and for the 
Navier-Stokes equations e.\ g.\ to the work of Galdi \cite{Gal:11}, Farwig, Kozono and Sohr \cite{FKS:05}, Abe and Giga \cite{AG:13} and Lemari\'e-Rieusset \cite{LR:16}.

Lions, Temam and Wang proved in their pioneering work the existence of a weak solution to the primitive equations and to the CAO-system; its uniqueness
remains an open problem for  both systems until today. 

On the other hand, the primitive equations subject to standard Dirichlet and Neumann boundary conditions are known to be globally strongly well-posed in the three-dimensional setting  
for arbitrarily  large data belonging to $\rH^1$ by the celebrated result of Cao and Titi \cite{CT:07}. For related results, see \cite{KZ:07,HK:16,GGHHK:20b}.     

It is the aim of this article to show  that the coupled CAO-system is globally well-posed in the strong sense for large data, even in critical Besov spaces. Recalling that uniqueness of  
weak solutions constructed by Lions, Temam and Wang is not known, it is hence surprising that we are able to prove that for initial data even in critical Besov spaces of the form 
$\rB^{2/q}_{qp}$, the CAO-system admits a {\em unique, global, strong} solution within the $\rH^{1,p}(0,T;\rL^q) \cap \rL^p(0,T;\rH^{2,q})$-framework. It holds true in particular for the case $p=q=2$ and hence 
in the $\rH^1(0,T;\rL^2) \cap \rL^2(0,T;\rH^2)$-situation.  In the absence of outer forces, the solution is even real analytic away from the boundary.  Our result described  in 
\autoref{thm:globalwellposed} below gives hence an answer to a long-standing open problem in fluid dynamics.          

At this point, some words about the CAO-model are in order. Very roughly speaking, it couples the compressible primitive equations describing the dynamics of the  atmosphere and 
the incompressible primitive equations for the dynamics of the ocean through a fully nonlinear  interface condition. Lions, Temam and Wang introduced in their analysis of the 
compressible equation for the atmosphere a new $p$-coordinate system, which ``resembles as far as the mathematical structure of the equations is concerned that of an incompressible fluid''.
In fact, assuming the pressure of the atmosphere at the interface to be constant, one obtains in this way  essentially an incompressible primitive equation.      

The physical law describing the driving mechanism on the interface is the  balance of the shear stress of the fluid and the horizontal component 
of the atmosphere force. The shear stress of the ocean, i.\ e., the tangential component of the stress tensor, is given by $\partial_z v^{\mathrm{o}} + \nablaH w$, which due to the flatness of the interface 
equals $\partial_z v^{\mathrm{o}}$. On the other hand, the horizontal atmospheric force is given as a drag force of the form $ c\varrho^{\mathrm{a}} (v^{\mathrm{a}}-v^{\mathrm{o}})|v^{\mathrm{a}}-v^{\mathrm{o}}|$. 

Observing that two-phase flows with  linear coupling conditions are rather  well understood, we note that this is not the case for fully nonlinear interface conditions as formulated in CAO.     
   
Let us now roughly describe the coupling conditions within the CAO-system. We refer the reader to \cite{LTW:95} for a detailed discussion.  
Assuming the interface to be fixed and flat, and following the physical law described above, they read as 
\begin{equation}\label{eq:coupling}   
c_1 \varrho^{\mathrm{a}} \partial_z v^{\mathrm{o}} = - c_2 \varrho^{\mathrm{a}} \partial_z  v^{\mathrm{a}} = c \varrho^{\mathrm{a}}(v^{\mathrm{a}}-v^{\mathrm{o}})|v^{\mathrm{a}}-v^{\mathrm{o}}|.
\end{equation}
Here  $v^{\mathrm{o}}, v^{\mathrm{a}}$ denote the horizontal velocities of the ocean and the atmosphere, respectively, $\varrho^{\mathrm{a}}, \varrho^{\mathrm{o}}$ their densities, $c_1,c_2$ are physical constants, and $c$  is a 
coefficient arising from boundary layer theory. 

One might suspect that a natural boundary condition on the interface between two immiscible viscous fluids  would be an adherence condition, i.\ e., $v^{\mathrm{o}}=v^{\mathrm{a}}$.  
This condition, however,  cannot be used due to the occurrence of boundary layers in the atmosphere and in the ocean at the interface. The atmospheric boundary layer may be divided into two 
horizontal layers, the Prandtl layer and the planetary boundary layer. Following Lions, Temam and Wang \cite{LTW:95}, the above condition \eqref{eq:coupling} takes into account 
these boundary layers by assuming that they are geometrically replaced by an interface $\Gamma$ whose upper side corresponds to the top of the Prandtl layer and the 
lower side to the bottom of the ocean boundary layer. For more information on these boundary layers, we refer e.\ g.\ to the work of Dalibard and G\'erard-Varet \cite{DGV:17} or Dalibard and 
Saint-Raymond \cite{DS:09}.

To avoid the difficulty of fully nonlinear interface conditions, they are often  replaced by simplified linear conditions as $\partial_z \vocn = |\vatm| \cdot \vatm$ for a given function $\vatm$.
As pointed out in \cite{LTW:93} and \cite{BS:01}, this modification is unjustified and generates unrealistic solutions from a physical point of view.
For a discussion of these boundary conditions, we refer e.\ g.\ to the work of Bresch and Simon  \cite{BS:01}. For related results about the Navier-Stokes equations with wind-driven boundary conditions, we 
refer to the work of Desjardins and Grenier \cite{DG:00}, Bresch and Simon \cite{BS:01}, Bresch, Guill\'en-Gonz\'alez, Masmoudi and Rodr\'iguez-Bellido \cite{BGMR:03} and Dalibard and Saint-Raymond \cite{DS:09}.
The primitive equations with stochastic wind-driven boundary conditions, where $\vatm$ is a given stochastic process, were analyzed recently in \cite{BHHS:23}.   

In contrast to the existing results in the literature, our aim is not to  simplify the interface condition,  but to consider the original CAO-system subject  to fully 
nonlinear interface conditions.

Concerning  the uniqueness problem of weak solutions for a  single primitive equations, several approaches have been developed in the last years aiming for extending well-posedness results 
for the primitive equations  to the case of rough initial data. We refer here to the  work \cite{LT:16}, yielding uniqueness of  weak solutions under special assumptions on the data.
For the existence of unique, strong, global solutions for rough initial data of the form $u_0 = a_1 + a_2$, where $a_1$ is continuous and $a_2$ is a small $\rL^\infty$ perturbation, see \cite{GGHHK:20b}. 
Regardless of these efforts, the uniqueness problem for weak solutions to the primitive equations remains unsolved.  

Our approach to the CAO-system can be described as follows: We first develop an optimal data result  for the linearized CAO-problem  within the  $\rL^p$-$\rL^q$-framework, which is of 
independent interest. The key observation here is the fact that the linearized problem admits a unique, strong solution in the  $\rL^p$-$\rL^q$-setting if and only if the boundary terms belongs to 
certain vector-valued Triebel-Lizorkin spaces, more precisely to $\rF^s_{pq}$, where $s=1-\frac{1}{2q}$ for the Dirichlet part and $s=\frac{1}{2}-\frac{1}{2q}$ for the Neumann part of the boundary. 
We then verify this condition by weighted paraproduct methods in these Triebel-Lizorkin spaces.  Combining this with  $\rL^p$-$\rL^q$-estimates for the nonlinearities as in  \cite{HK:16}, we obtain local well-posedness as well as a blow-up criterion. 
Let us emphasize that the existence and uniqueness of strong solutions are {\em characterized} by the 
boundary terms belonging to these Triebel-Lizorkin spaces, so there is no choice but to treat the fully nonlinear boundary terms in these spaces. 

In a second step, we deduce a-priori estimates for the fully nonlinear  problem taking  into account the nonlinear interface terms. 
In order to exploit these a-priori estimates, it is crucial that the linear theory also covers the Hilbert space case. 
To the best of our knowledge, this can only be achieved by means of an optimal data approach as sketched in the previous paragraph. 
We note that the signs appearing in the boundary terms play here an 
important role. These estimates then  rule out the finite-in-time blow-up scenario, and we obtain the  global, strong well-posedness of the CAO-system. Our  approach yields in addition the global 
strong well-posedness of the CAO-system for initial data in critical Besov spaces. Compared  to existing approaches to a-priori estimates \cite{CT:07, HK:16} for a 
single system of primitive equations, the situation for CAO is seriously more involved due to fully nonlinear interface conditions. 

This article is structured as follows. In \autoref{sec:prelim}, we recall the precise formulation of the CAO-model from \cite{LTW:93}. In 
\autoref{sec:main}, we state  our main theorem on the global strong well-posedness of the CAO-problem for initial data in critical Besov spaces and discuss in addition 
regularity properties of the solutions.  \autoref{sec:linearizedproblem} analyzes the inhomogeneous hydrostatic Stokes problem, and the optimal $\rL^p$-$\rL^q$-data result is presented. 
Paraproduct techniques are used to estimate in  \autoref{sec:local+blowup} the boundary terms in Triebel-Lizorkin spaces. Combing these results, we obtain also in \autoref{sec:local+blowup} local 
well-posedness as well as a blow-up criterion. Finally, in \autoref{sec:global}, we establish $\rL^\infty_t \rH^1 \cap \rL^2_t \rH^2$-bounds, yielding a-priori estimates in the maximal regularity 
space and hence global strong well-posedness.  
  
We denote by $\nablaH v = (\partial_x v,\partial_y v)^\top$, $\divH v = \partial_x v_1 + \partial_y v_2$ and $\DeltaH = \partial_x^2 + \partial_y^2$ the horizontal gradient, the horizontal divergence and 
the horizontal Laplacian, respectively, and by $\mdiv$, $\nabla$ as well as $\Delta$ the divergence, gradient and Laplacian in either $p$- or in $z$-coordinates.

\vspace{-.2cm}
\section{Preliminaries}\label{sec:prelim}
Lions, Temam and Wang introduced in \cite{LTW:93} the CAO-system, consisting of a compressible primitive equation for the atmosphere and an incompressible primitive equation 
for the ocean which are coupled by fully nonlinear interface conditions. The system is modeled on a cylindrical domain taking the precise shape~$\Omegaair \ \dot \cup \ \Omegaocn = G \times (0,1) \  \dot \cup \ G \times (-1,0)$, where 
$G \in \{ \R^2, \T^2 \}$. The interface is given by $\Gamma_i = G \times \{ 0 \}$. In the isothermal situation, and with thermodynamic constants set equal to~$1$, the equation 
for the atmosphere takes the shape
\begin{equation}
	\left\{
	\begin{aligned}
        \dt \varrho^{\mathrm{a}} + \mdiv (\varrho^{\mathrm{a}} \uair)  &=0, &&\text{ on }\Omegaair \times (0,T), \\
		\dt  \vair - \Delta \vair + \uair \cdot \nabla \vair + \frac{1}{\varrho^{\mathrm{a}}}\nablaH \pair &= \fair, &&\text{ on } \Omegaair \times (0,T),  \\
		  \dz \pair  &= -\varrho^{\mathrm{a}} g , &&\text{ on } \Omegaair \times (0,T), \\ 
        \pair &=  \varrho^{\mathrm{a}} , &&\text{ on } \Omegaair \times (0,T),  
	\end{aligned}
	\right. 
	\label{eq:compressibleprimitive}
\end{equation}
where $\uair = (\vair, \wair) : \Omegaair \rightarrow \R^3$ denotes the velocity, $\pair : \Omegaair \rightarrow \R$ is the atmospheric pressure, $\varrho^{\mathrm{a}} : \Omegaair \rightarrow \R$ represents the 
atmospheric density, and $\fair : \Omegaair \rightarrow \R^2$ denotes an external force.

The ocean dynamics is  described by the following incompressible primitive equations
\begin{equation}
	\left\{
	\begin{aligned}
		\dt  \vocn - \Delta \vocn + \uocn \cdot \nabla \vocn + \nablaH \pi &= \focn, &&\text{ on } \Omegaocn \times (0,T),  \\
		  \dz \pi  &= - g , &&\text{ on } \Omegaocn \times (0,T), \\ 
		\mdiv \uocn &= 0, &&\text{ on } \Omegaocn \times (0,T),
	\end{aligned}
	\right. 
	\label{eq:incompressibleprimitive}
\end{equation}
where for simplicity, we assume that $\varrho^o$ is constantly ~$1$. In this case, the unknowns are the velocity of the ocean~$\uocn = (\vocn , \wocn) : \Omegaocn \rightarrow \R^3$ as well as the pressure of the ocean $\pi : \Omegaocn \rightarrow \R$. As above, $\focn : \Omegaocn \rightarrow \R^2$ denotes an external forcing term. 

The equations \eqref{eq:compressibleprimitive} and \eqref{eq:incompressibleprimitive} are coupled on the interface $\Gamma_i =G \times \{ 0 \}$ by the balance law saying that the shear stress of the 
fluid equals the horizontal component of the atmosphere force, i.\ e.,
\begin{equation*}
    c_1 \varrho^{\mathrm{a}} \dz \vatm = -c_2 \dz \vocn = c \varrho^{\mathrm{a}} (\vatm - \vocn) | \vatm - \vocn|.
\end{equation*}
This equation  is to be understood in the sense of traces. The coupled system is further supplemented by the boundary conditions
\begin{equation*} 
\begin{aligned} 
    (\partial_z \vatm)|_{G \times \{1 \}} &= 0 &&\text{ and } \quad \wair|_{G \times \{ 0\} \cup G \times \{ 1 \}} = 0 , \\ 
    (\partial_z \vocn)|_{G \times \{-1 \}} &= 0 &&\text{ and } \quad  \wocn|_{G \times \{ -1\} \cup G \times \{ 0 \}} = 0 .
    \end{aligned} 
\end{equation*}
In order to deal with the compressible equation, the atmospheric pressure $\pair $ is introduced as a new independent variable. The transformation reads as
\begin{equation*}
    (t,x,y,z) \rightarrow (t,x,y,p(t,x,y,z)), \text{ with inverse } (t,x,y,p) \rightarrow (t,x,y,z(t,x,y,p)).
\end{equation*}
We observe that in this coordinate system, the domain is a moving domain and no longer a fixed  one. 
In fact,  the transformed domain $\Omegaair(t)$ is given by
\begin{equation*}
    \Omegaair(t) := \left \{ (x,y,p) : (x,y) \in G, \ p \in ( \pair_s(t) \mathrm{e}^{-g}, \pair_s(t)) \right \},
\end{equation*}
where $\pair_s(t)$ denotes the pressure of the atmosphere at the surface depending on time and the horizontal coordinates. Introducing the geopotential $\Tilde{\Phi}$ and the material derivative of the pressure $\omega$ as new quantities, i.\ e., setting
\begin{equation*}
    \Tilde{\Phi} (t,x,y,p) := g z(t,x,y,p) \quad \text{ and } \quad \omega := \dt \pair + \uair \cdot \nabla \vair, 
\end{equation*}
we obtain the equations for the atmosphere of the shape
\begin{equation}
	\left\{
	\begin{aligned}
		\dt  \vair  - L\vair  + \vair \cdot \nablaH \vair + \omega \partial_p \vair + \nablaH \Tilde{\Phi} &= \fair, &&\text{ on } \Omegaair(t) \times (0,T),  \\
		  \partial_p \Tilde{\Phi}   &= -\frac{1}{p} , &&\text{ on } \Omegaair(t) \times (0,T), \\ 
        \divH \vair + \partial_p \omega &= 0, &&\text{ on } \Omegaair(t) \times (0,T), 
	\end{aligned}
	\right. 
	\label{eq:compressibleprimitivepcoordinates}
\end{equation}
where $L$ denotes the transformed Laplacian made precise below. The boundary conditions transform to
\begin{equation*} 
\begin{aligned} 
    (\pair_s)^2(\partial_p \vatm)|_{G \times \{\pair_s \}} &=c\pair_s (\vatm - \vocn) | \vatm - \vocn| &&\text{ and } \quad \omega|_{G \times \{ \pair_s\}} =  \dt \pair_s + \vair \nablaH \pair_s , \\ 
    (\partial_p \vatm)|_{G \times \{ \pair_s \mathrm{e}^{-g}\}} &= 0 ,&&\text{ and } \quad  \omega|_{G \times \{ \pair_s \mathrm{e}^{-g} \} } =   \mathrm{e}^{-g} \left (\dt \pair_s + \vair \nablaH \pair_s \right ) .
    \end{aligned} 
\end{equation*}
From \eqref{eq:compressibleprimitivepcoordinates}$_2$ we deduce that $\Tilde{\Phi}(t,x,y,p) =\log(\pair_s) -\log(p)$.
The equation \eqref{eq:compressibleprimitivepcoordinates} thus simplifies to
\begin{equation*}
	\left\{
	\begin{aligned}
		\dt  \vair  - L\vair  + \vair \cdot \nablaH \vair + \omega \partial_p \vair + \nablaH (\log \pair_s )&= \fair, &&\text{ on } \Omegaair(t) \times (0,T),  \\
        \divH \vair + \partial_p \omega &= 0, &&\text{ on } \Omegaair(t) \times (0,T), 
	\end{aligned}
	\right. 
\end{equation*}
and the operator $L$ is given by
\begin{equation*}
 L\vair:= \DeltaH \vair + \partial_p ((pg)^2 \partial_p \vair) - 2 \frac{p}{\pair_s} \nablaH \partial_p \vair \nablaH \pair_s - \Big(\frac{p |\nablaH \pair_s |}{\pair_s} \Big)^2 \partial_p^2 \vair  
              + p \partial_p \vair \DeltaH \pair_s -2p\Big( \frac{|\nablaH \pair_s |}{\pair_s} \Big)^2 \partial_p\vair. 
\end{equation*}
Lions, Temam and Wang \cite{LTW:93} assumed the atmospheric pressure on the interface to be constant, meaning that $\pair_s \equiv p_s >0$. This assumption has important consequences. Indeed, 
the moving domain $\Omegaair(t)$ is changed  to a fixed domain $\Omegaair = \{ (x,y,p) : (x,y) \in G, \ p \in (p_s \mathrm{e}^{-g}, p_s) \}$, and the transformed Laplacian $L$ simplifies to the operator 
$\DeltaH + \partial_p ((pg)^2 \partial_p)$. Also, they presumed that the geopotential $\tilde{\Phi}$ is adjusted by an averaged function, depending only on the first two independent variables, 
such that the geopotential at the interface  recovers the topography of the earth. We will denote the new geopotential by $\Phi$. For a detailed physical discussion of those assumptions, we 
refer to \cite{LTW:92a} and \cite{LTW:93}. 
The system under consideration hence is 
    \begin{equation}
	\left\{
	\begin{aligned}
		\dt \vair  + \uair \cdot \nabla \vair + \nablaH \Phi &= f^{\air}+ \DeltaH \vair + \partial_p \left (p^2 \partial_p \vair\right ), &&\text{ on } \Omegaair \times (0,T),  \\
		  \partial_p \Phi  &= -\frac{1}{p} , &&\text{ on } \Omegaair \times (0,T), \\ 
		\mdiv \uatm &= 0, &&\text{ on } \Omegaair \times (0,T),\\
        \dt \vocn  + \uocn \cdot \nabla \vocn + \nablaH \pi &= f^{\ocn}+ \Delta \vocn, &&\text{ on } \Omegaocn \times (0,T), \\
        \dz \pi &= -1, &&\text{ on } \Omegaocn \times (0,T), \\
        \mdiv \uocn &= 0, &&\text{ on } \Omegaocn \times (0,T), 
	\end{aligned}
	\right. 
	\label{eq:CAO}
	\tag{CAO} 
\end{equation}
where we assume $g$ to be $1$. In the sequel, $\Gammaairu$, $\Gammaairb$ and $\Gau$, $\Gab$ denote the upper and lower boundaries of the atmosphere and ocean domains, respectively. 
Note that the upper boundary $\Gammaairu = G \times \{ p_s \}$ corresponds to the upper boundary of the ocean $\Gau = G \times \{ 0 \}$. 
The boundary conditions of \eqref{eq:CAO} then read as
\begin{equation} 
\begin{aligned} 
    (\partial_p \vatm)|_{\Gammaairb} &= 0, \quad &&(\partial_p \vatm)|_{\Gammaairu} = -p_s^{-1} V|V| &&&\text{and } \quad \omega|_{\Gammaairb \cup \Gammaairu} = 0 , \\ 
    (\partial_z \vocn)|_{\Gab} &= 0, \quad &&(\partial_z \vocn)|_{\Gammaairu} = p_s V|V|   &&&\text{and } \quad  w|_{\Gau \cup \Gab} = 0 ,
    \end{aligned} 
    \label{eq:bc}
\end{equation}
with $V$ representing the relative velocity evaluated at the interface, i.\ e.,
\begin{equation}\label{eq:V}
	V:= \vair|_{\Gammaairu}-\vocn|_{\Gau}.
\end{equation}
The system is completed by the initial data
\begin{equation}
    (\vair, \vocn)|_{t=0} = (\vair_0, \vocn_0).
    \label{eq:initial data}
\end{equation}

Next, \eqref{eq:CAO}$_2$ and \eqref{eq:CAO}$_5$ imply
\begin{equation*}
    \begin{aligned}
        \Phi(x,y,p)  = \Phi_s(x,y) + \log(p_s) -\log(p) \quad \text{ and } \quad
        \pi(x,y,z) = \pi_s(x,y) - z.
    \end{aligned}
\end{equation*}
Here we denote by $\Phi_s$ and $\pi_s$ the geopotential and oceanic pressure evaluated at the interface. Observe that $\nablaH \Phi = \nablaH \Phi_s$ and $\nablaH \pi = \nablaH \pi_s$. 
We then conclude that \eqref{eq:CAO} can be rewritten equivalently as
\begin{equation}
	\left\{
	\begin{aligned} \dt \vair+ \vair \cdot \nablaH \vair + \omega \partial_p \vair  + \nablaH \Phi_s &= \fair+\Delta^{\atm} \vair, &&\text{ on } \Omegaair \times (0,T),  \\
		\divH \vairbar &= 0, &&\text{ on } \Omegaair \times (0,T),\\
        \dt \vocn  + \vocn \cdot \nablaH \vocn + w \dz \vocn + \nablaH \pi_s &= f^{\ocn}+\Delta \vocn, &&\text{ on } \Omegaocn \times (0,T), \\
        \divH \vocnbar &= 0, &&\text{ on } \Omegaocn \times (0,T), 
	\end{aligned}
	\right. 
	\label{eq:CAOsimplified}
	\tag{CAO*}
\end{equation}
with $\Delta^{\air} := \DeltaH + \partial_p (p^2 \partial_p )$, and complemented by the boundary conditions
\begin{equation*} 
\begin{aligned} 
    (\partial_p \vatm)|_{\Gammaairb} = 0 \quad &\text{ and } \quad (\partial_p \vatm)|_{\Gammaairu} = -p_s^{-1} V|V| , \\ 
    (\partial_z \vocn)|_{\Gab} = 0 \quad &\text{ and } \quad (\partial_z \vocn)|_{\Gau} = p_s V|V|
    \end{aligned} 
\end{equation*}
as well as the initial data
\begin{equation*}
    (\vair, \vocn)|_{t=0} = (\vair_0, \vocn_0) .
\end{equation*}

Some words about our setting for solving the CAO-problem  are in order.  Let $s \in \R$ and $p,q \in (1,\infty)$. In the sequel, we denote 
by $\rW^{s,q}(\Omega)$ the fractional Sobolev spaces, 
by $\rH^{s,q}(\Omega)$ the Bessel potential spaces, by~$\rB^{s}_{qp}(\Omega)$ the Besov spaces, and by $\rF_{pq}^s(\Omega)$ the Triebel-Lizorkin spaces, where $\Omega \subset \R^3$ is open.  As usual, we set $\rH^{0,q}(\Omega) := \rL^p(\Omega)$ and note that $\rH^{s,q}(\Omega)$ coincides with the 
classical Sobolev space $\rW^{m,q}(\Omega)$ if $s=m \in \N$. For details concerning these function spaces, we refer to the monographs \cite{Tri:78, BCD:11, Ama:19}. 

 Next, we introduce time-weighted Banach space-valued spaces $\rL^p_\mu(J;E)$. For a Banach space $(E,\| \cdot \|_E)$, a time interval $J = (0,T)$, with $0 < T \le \infty$, as well as $p \in (1,\infty)$, $\mu \in (\nicefrac{1}{p},1]$ and $u \colon J \to E$, we denote by~$t^{1-\mu} u$ the function $t \mapsto t^{1-\mu} u(t)$ on $J$.
Moreover, we define
\begin{equation*}
    \rL_\mu^p(J;E) := \left\{u \colon J \to E : t^{1-\mu} u \in \rLp(J;E)\right\}.
\end{equation*}
The norm
$\| u \|_{\rL_\mu^p(J;E)} := \| t^{1-\mu} u  \|_{\rLp(J;E)} = \left(\int_J t^{p(1-\mu)} \| u(t) \|_E^p \d t\right)^{\nicefrac{1}{p}}$ renders the latter space a Banach space. 
For $k \in \mathbb{N}_0$, the associated weighted Sobolev spaces are defined by
\begin{equation*}
    \rH_\mu^{k,p}(J;E) := \bigl\{u \in \rH_\loc^{k,1}(J;E) : u^{(j)} \in \rL_\mu^p(J;E), \enspace j \in \{0,\dots,k\}\bigr\},
\end{equation*}
and these spaces are equipped with the norms
$
    \| u \|_{\rH_\mu^{k,p}(J;E)} := \bigl(\sum_{j=0}^k \| u^{(j)} \|_{\rL_\mu^p(J;E)}^p \bigr)^{\nicefrac{1}{p}}
$. 
The weighted fractional Sobolev spaces and Bessel potential spaces are then defined by interpolation.

Since vector-valued Triebel-Lizorkin spaces turn out to be  characteristic for boundary values in our approach, we give here a precise definition.  
To this end, for $\mathscr{S}(\R)$ denoting the space of Schwartz functions on $\R$, let $\varphi \in \mathscr{S}(\R)$ with $0 \leq \hat{\varphi}(\xi) \leq 1$ for $\xi \in \R$, and $\hat{\varphi}(\xi) = 1$ if $ |\xi|\leq 1$ as well as 
$ \hat{\varphi}(\xi) = 0$ for $|\xi| \geq \frac{3}{2}$. 
Then let $\hat{\varphi_0} = \hat{\varphi}$, $\hat{\varphi_1}(\xi)= \hat{\varphi}(\nicefrac{\xi}{2})- \hat{\varphi}(\xi)$ and
$    \hat{\varphi_k}(\xi) = \hat{\varphi_1}(2^{-k+1}\xi) = \hat{\varphi}(2^{-k}\xi) - \hat{\varphi}(2^{-k+1}\xi)$ for $\xi \in \R$ and $k \geq 1$.
For $\varphi$ as above and $f \in \mathscr{S}'(\R,E)$, define
 $   S_k f := \varphi_k \ast f = \mathscr{F}^{-1}(\hat{\varphi_k}\hat{f})$,
which belongs to $\rC^{\infty}(\R;E) \cap \mathscr{S}'(\R;E)$. It follows that $\sum_{k \geq 0} S_k f =f$ in the sense of distributions since $\sum_{k\geq 0} \hat{\varphi_k}(\xi)=1$ for all $\xi \in \R$. 
If  $E$, $p$, $q$ and $s$ are as above, then the time-weighted Triebel-Lizorkin space $\rF^s_{pq,\mu}(\R,E)$ is defined by 
\begin{equation*}
    \begin{aligned}
         \rF^s_{pq,\mu}(\R,E) 
         &:= \bigl\{ f \in \mathscr{S}'(\R,E) \colon \| f \|_{\rF^s_{pq,\mu}(\R,E)} < \infty \bigr\}, \quad \text{ where}\\
         \| f \|_{\rF^s_{pq,\mu}(\R,E)} 
         &:= \big \| \left ( 2^{ks} S_k f \right )_{k \geq 0}\big \|_{\rL^p_{\mu}(\R; l^q(E))}.
    \end{aligned}
\end{equation*}
The space $\rF^s_{pq,\mu}(J,E)$ is defined via restriction.
Next, we define \emph{hydrostatically solenoidal vector fields} by
\begin{equation*}
    \rL^q_{\sigmabar}(\Omega) = \overline{\{v \in \rC^\infty(\overline{\Omega};\R^2) \colon \divH \vbar = 0 \}}^{\| \cdot \|_{\rL^q(\Omega)}}
\end{equation*}
on $\Omega \in \{ \Omegaatm, \Omegaocn \}$. Its role for the primitive equations parallels the one of the solenoidal vector fields for the Navier-Stokes equations. 
For a function $f \colon G \times (a,b) \rightarrow \R^2$,we define here the vertical average by~$\overline{f} := \frac{1}{b-a} \int_{a}^{b} \divH f(\cdot , \cdot , \xi) \d \xi$.

Last, we introduce the spaces for the initial data.
In fact, for $p, q \in (1,\infty)$ such that $\frac{1}{p} + \frac{1}{q} \leq 1$, we set
\begin{equation*}
    \rB_{qp,\sigmabar}^{\nicefrac{2}{q}}(\Omegaatm) \coloneqq \rB_{qp}^{\nicefrac{2}{q}}(\Omegaatm;\R^2) \cap \rL_{\sigmabar}^q(\Omegaatm) \quad \text{ and } \quad \rB_{qp,\sigmabar}^{\nicefrac{2}{q}}(\Omegaocn) \coloneqq \rB_{qp}^{\nicefrac{2}{q}}(\Omegaocn;\R^2) \cap \rL_{\sigmabar}^q(\Omegaocn).
\end{equation*}

\section{Main result}\label{sec:main}

We are now in the position to formulate our main result on the global, strong well-posedness of \eqref{eq:CAO} for arbitrarily large data in critical Besov spaces.

\begin{thm}[Global strong well-posedness for the CAO-system]\label{thm:globalwellposed} \mbox{} \\
Let $1 < q \leq p < \infty$ such that $\mu_c = \frac{1}{p} + \frac{1}{q} \leq 1$ and $T > 0$. Assume that $(\vair_0, \vocn_0) \in \rB_{qp,\sigmabar}^{\nicefrac{2}{q}}(\Omegaatm) \times 
\rB_{qp,\sigmabar}^{\nicefrac{2}{q}}(\Omegaocn)$ as well as $(\fair, \focn) \in \rLp_{\mu_c}(0,T;\rLq(\Omegaair;\R^2)) \times \rLp_{\mu_c}(0,T;\rLq(\Omegaocn;\R^2))$. 

Then there exists a unique, strong solution $(\vair,\omegaair,\Phi)$, $(\vocn, \wocn, \pi)$ to \eqref{eq:CAO}, subject to \eqref{eq:bc} and \eqref{eq:initial data}, satisfying
\begin{equation*}
		\begin{aligned}
			\vair &\in 
            \rH^{1,p}_{\mu_c}(0,T;\rLq(\Omegaair;\R^2)) \cap \rLp_{\mu_c}(0,T;\rH^{2,q}(\Omegaair;\R^2)) \cap \rC([0,T];\rB_{qp,\sigmabar}^{\nicefrac{2}{q}}(\Omegaatm)), \\
            \vocn &\in 
            \rH^{1,p}_{\mu_c}(0,T;\rLq(\Omegaocn;\R^2)) \cap \rLp_{\mu_c}(0,T;\rH^{2,q}(\Omegaocn;\R^2)) \cap \rC([0,T];\rB_{qp,\sigmabar}^{\nicefrac{2}{q}}(\Omegaocn)).
	   \end{aligned}
\end{equation*}
\end{thm}

\begin{rem}{\rm 
(i) An analogous result for the simplified problem concerning linear interface conditions, so $\partial_p \vair = p_s^{-1} \cdot (\vocn-\vatm)$ and 
$\dz \vocn = p_s \cdot (\vatm-\vocn)$ on $\Gammaairb \cong \Gau$, can be obtained  more easily.  The estimates for  the boundary terms are then much less involved and can be deduced 
in a more straightforward way.\\
(ii) The assertion of \autoref{thm:globalwellposed} remains valid when considering general time weights $\mu \in [\nicefrac{1}{p} + \nicefrac{1}{q},1]$. In that case, the space for the initial data is given 
by $\rB_{qp,\mathrm{N},\sigmabar}^{2(\mu-\nicefrac{1}{p})}(\Omegaatm) \times \rB_{qp,\mathrm{N},\sigmabar}^{2(\mu-\nicefrac{1}{p})}(\Omegaocn)$, where the subscript $_\mathrm{N}$ indicates Neumann boundary conditions on the 
upper and lower boundary, respectively, if $\mu > \frac{1}{2} + \frac{1}{p} + \frac{1}{2q}$. The initial data then also need to satisfy the nonlinear coupling conditions on the interface 
as in \eqref{eq:bc}.
}
\end{rem}

Our  next result discusses regularity properties of the solution (in time and space) under stronger regularity assumptions on the data. Interestingly enough, it seems that only one additional time and spatial tangential derivative for the solution at the interface $\Gamma$ can be gained when presuming additional time and spatial regularity of the forces. This is 
 due to the fact that the interface  conditions are only once  Fr{\'e}chet-differentiable, but not twice.  

\begin{cor}[Regularity]\label{cor:regularity} \mbox{} \\
Let $p,q,\mu$, $v^{\atm}_0,v^{\ocn}_0$ and $f^\atm,f^\ocn$ be as in \autoref{thm:globalwellposed}.
Furthermore, let $(\vair, \vocn)$ be the unique, global, strong solution to \eqref{eq:CAO} given in  \autoref{thm:globalwellposed}.
\begin{enumerate}[(i)]
\item {\em Interior regularity:} Assume $q > \frac{5}{2}$,
$ \fair \in  \rC^k((0,T) \times \Omegaair)$ and $\focn \in \rC^k((0,T) \times \Omegaocn)$
for $k \in \N_0 \cup \{\infty\}$. Then there exists $\alpha \in (0,1)$, with $\rC^{\infty,\alpha}$ being identified with $\rC^\infty$, such that
\begin{equation*}
    \vair \in \rC^{k,\alpha}((0,T) \times \Omegaatm) \quad \text{ and } \quad 
        \vocn \in \rC^{k,\alpha}((0,T) \times \Omegaocn) .
\end{equation*}  
\item {\em Regularity at the Interface I:} 
Assume $\frac{2}{p}+\frac{3}{q} < 1$, $\fair \in  \rC^1((0,T) \times \overline{\Omegaair})$ and $\focn \in \rC^1((0,T) \times \overline{\Omegaocn})$.
Then 
\begin{equation*}
\vair, \, \nablaH \vair \in \rC^{1,\alpha}((0,T) \times \Gamma) \quad \text{ and } \quad 
        \vocn, \, \nablaH \vocn \in \rC^{1,\alpha}((0,T) \times \Gamma) \quad \mbox{ for } \quad \alpha \in (0, \nicefrac{1}{3}-\nicefrac{2}{3p}-\nicefrac{1}{q}).
\end{equation*} 
\item {\em Regularity at the Interface II:} 
Assume $\frac{2}{p}+\frac{3}{q} < 2$, $\fair \in  \rC^1((0,T) \times \overline{\Omegaair})$ and  $\focn \in \rC^1((0,T) \times \overline{\Omegaocn})$. Then 
\begin{equation*}
    \begin{aligned}
\vair \in \rC^{1,\alpha}((0,T) \times \Gamma) \quad \text{ and } \quad \vocn \in \rC^{1,\alpha}((0,T) \times \Gamma) \quad \mbox{ for } \quad \alpha 
\in (0, \nicefrac{2}{3}-\nicefrac{2}{3p}-\nicefrac{1}{q}).
 \end{aligned}
\end{equation*}
\end{enumerate}
\end{cor}

The next corollary exhibits the real analyticity of the solution in the interior with absent outer forces. 

\begin{cor}[Analyticity of the solution]\label{cor:regularitywithout} \mbox{} \\
Let $p,q,\mu,\alpha$, $v^{\atm}_0,v^{\ocn}_0$ be as in \autoref{cor:regularity}(i), and assume $f^\atm=f^\ocn=0$.  
In addition, let $(\vair, \vocn)$ be the unique, global, strong solution to \eqref{eq:CAO} given in  \autoref{thm:globalwellposed}.
Then $\vair$ and $\vocn$ are real analytic, i.\ e., 
\begin{equation*}
    \vair \in \rC^{\omega}((0,T) \times \Omegaatm) \quad \text{ and } \quad \vocn \in \rC^{\omega}((0,T) \times \Omegaocn). 
\end{equation*}
\end{cor}

\section{The linearized problem and optimal $\rL^p$ -$\rL^q$ theory}\label{sec:linearizedproblem}

In this section, we develop an optimal data result  for the linearized CAO-problem  within the  $\rL^p$-$\rL^q$-framework, which is of independent interest. 
The result serves as the starting point of the proof of the local well-posedness of the CAO-system. It is optimal in the sense that the assumptions are sufficient and necessary 
for solving the associated inhomogeneous linear problem given by
\begin{equation}
	\left\{
	\begin{aligned}
		\dt v - \Delta v + \lambda v + \nablaH \pi &= f_v, &&\on \Omega \times (0,T), \\
		\partial_z \pi &= f_w, &&\on \Omega \times (0,T), \\
		\mdiv u &= f_{\mdiv}, &&\on \Omega \times (0,T), \\
		v(0) &= v_0, &&\on \Omega \times (0,T), \\
		v &= b_v^{\D}, &&\text{ on } \Gamma_{\D} \times (0,T), \\
		\dz v &= b_v^{\mathrm{N}}, &&\text{ on } \Gamma_{\mathrm{N}} \times (0,T), \\
		w &= b_w^{\D}, &&\text{ on } \Gamma \times (0,T),
	\end{aligned}
	\right. 
	\label{eq:hstokes1}
\end{equation}
for $0 < T < \infty$ and $u = (v,w)$ subject to  periodic boundary conditions on the lateral boundary or on an infinite layer, i.\ e., on 
$\Omega = G \times (a,b)$ with $G \in \{ \T^2, \R^2 \}$. We will use the notation $\Gamma = \Gamma_{\D} \cup \Gamma_{\rN}$ for the Dirichlet and Neumann parts of the boundary.
The following assumptions are crucial for solving the boundary value problem \eqref{eq:hstokes1} subject to inhomogeneous data.

\begin{asu}\label{assu:data}
Given $f_v$, $f_w$, $f_{\mdiv}$, $v_0$, $b_v^{\D}$, $b_v^{\rN}$ and $b_w^{\D}$, with $b_w^{\D} = b_{w,a}^{\D}$ on $G \times \{a\}$ and $b_w^{\D} = b_{w,b}^{\D}$ on~$G \times \{b\}$, assume
\begin{enumerate}[(i)]
    \item $f_v - \int_a^z \nablaH f_w \d \xi \in \rLp_\mu(0,T;\rLq(\Omega;\R^2)) \eqqcolon \E_{0,\mu}$,
    \item $\frac{b_{w,a}^{\D} - b_{w,b}^{\D}}{b-a} + \fbar_{\mdiv} \in \rH^{1,p}_\mu(0,T;\Dot{\rH}^{-1,q}(G)) \cap \rLp_\mu(0,T;\rH^{1,q}(G)) \eqqcolon \mathbb{G}_{1,\mu}$,
    \item $v_0 \in \rB_{qp}^{2(\mu-\nicefrac{1}{p})}(\Omega;\R^2) \eqqcolon \rX_{\gamma,\mu}$,
    \item $b_v^{\D} \in \rF_{pq,\mu}^{1-\nicefrac{1}{2q}}(0,T;\rLq(G;\R^2)) \cap \rLp_\mu(0,T;\rB_{qq}^{2 - \nicefrac{1}{q}}(G;\R^2)) \eqqcolon \F_{\D,\mu}$ and\\ $b_v^{\rN} \in \rF_{pq,\mu}^{\nicefrac{1}{2}-\nicefrac{1}{2q}}(0,T;\rLq(G;\R^2)) \cap \rLp_\mu(0,T;\rB_{qq}^{1 - \nicefrac{1}{q}}(G;\R^2)) \eqqcolon \F_{\rN,\mu}$, as well as $v_0 = b_v^{\D}(0)$ provided $1 - \frac{1}{2q} > \frac{1}{p} + 1 - \mu$, while $\dz v_0 = b_v^{\rN}(0)$ if $\frac{1}{2} - \frac{1}{2q} > \frac{1}{p} + 1 - \mu$, and finally
    \item $\divH \vbar_0 = \frac{b_{w,a}^{\D}(0) - b_{w,b}^{\D}(0)}{b-a} + \fbar_{\mdiv}(0)$ in $\sD'(G)$, i.\ e., in the sense of distributions.
	\end{enumerate}
\end{asu}

The main result in this section addresses the solvability of \eqref{eq:hstokes1}.

\begin{thm}\label{thm:optdata}
Let $p,q \in (1,\infty)$ and $\mu \in (\nicefrac{1}{p},1]$ be such that  $\mu \neq \frac{1}{p} + \frac{1}{2q}$ and $\mu \neq \frac{1}{2} + \frac{1}{p} + \frac{1}{2q}$.
Then there exists $\lambda \ge 0$ such that the problem \eqref{eq:hstokes1} admits a unique, strong solution $v$ satisfying
\begin{equation*}
    \begin{aligned}
        v &\in \E_{1,\mu} := \rH^{1,p}_\mu(0,T;\rLq(\Omega;\R^2)) \cap \rLp_\mu(0,T;\rH^{2,q}(\Omega;\R^2))
    \end{aligned}
\end{equation*}
if and only if the data $(f_v,f_w,f_{\mdiv},v_0,b_v^{\D}, b_v^{\rN}, b_w^{\D})$ satisfy \autoref{assu:data}.
In that case, there exists a constant~$C > 0$ such that $v$ satisfies the estimate
\begin{equation*}
    \| v \|_{\E_{1,\mu}} \le C\bigl(\bigl\| f_v - \int_a^z \nablaH f_w \d \xi \bigr\|_{\E_{0,\mu}} + \| b_{w,a}^{\D} - b_{w,b}^{\D}  +  \fbar_{\mdiv} \|_{\mathbb{G}_{1,\mu}} + \| v_0 \|_{\rX_{\gamma,\mu}} + 
\| b_v^{\D} \|_{\F_{\D,\mu}} + \| b_v^{\rN} \|_{\F_{\rN,\mu}}\bigr).
\end{equation*}
\end{thm}

\begin{rem}\label{rem:main:theorem}{\rm
(i) The vertical velocity $w$ and the pressure term $\pi$ are omitted in the above theorem. Indeed,~$w$ can be  determined by the divergence condition $\mdiv u=f_{\mdiv}$ and the boundary data 
$b_w^{\D}$, see \eqref{eq:w} below.
The pressure $\pi$ can be recovered from \eqref{eq:stokes} by applying the horizontal divergence to the equation and solving the resulting 
Poisson equation for the surface pressure. We refer here also to \eqref{eq:pressure}. \mbox{}\\
(ii) The shift by $\lambda \ge 0$ is generally indispensable if $\GaD = \emptyset$, i.\ e., for pure Neumann boundary conditions. 
}
\end{rem}

With regard to solving \eqref{eq:CAO}, the most relevant case corresponds to $f_w = f_{\mathrm{div}}= 0$, $b^{\D}_{w,a} = b^{\D}_{w,b} = 0$ and pure Neumann boundary conditions for $v$.
Hence, by \eqref{eq:hstokes1}$_2$, the pressure $\pi$ is fully determined by the surface pressure $\pi_s$.
In this particular situation, we get the following result.

\begin{cor}\label{cor:main}
Let $p,q,\mu$ be as in \autoref{thm:optdata}, and assume that $f_w = 0$ and $b^{\D}_{w,a} = b^{\D}_{w,b} = 0$.
Then there is $\lambda \ge 0$ such that \eqref{eq:hstokes1} admits a unique, strong solution $(v,\pi_s)$ with $v \in \E_{1,\mu}$ and $\pi_s \in \rL_\mu^p(0,T;\Dot{\rH}^{1,q}(G))$
if and only if $f_v \in \E_{0,\mu}$, $v_0 \in \rX_{\gamma,\mu}$ and $b_v^{\rN} \in \F_{\rN,\mu}$ fulfill \autoref{assu:data}(iv).
Moreover, there exists a constant~$C > 0$ such that $v$ fulfills the estimate
\begin{equation*}
    \| v \|_{\E_{1,\mu}} \le C\bigl(\| f_v \|_{\E_{0,\mu}} + \| v_0 \|_{\rX_{\gamma,\mu}} + \| b_v^{\rN} \|_{\F_{\rN,\mu}}\bigr).
\end{equation*}
\end{cor}

As our  proof of \autoref{thm:optdata} relies on a splitting of \eqref{eq:hstokes1} in its barotropic and baroclinic mode, we start by investigating  the perturbed heat equation with 
inhomogeneous data given by
\begin{equation}
	\left\{
	\begin{aligned}
		\dt V - \Delta V +\lambda V + B V &= H, &&\on \Omega \times (0,T),\\
		V(0) &= V_0, &&\on \Omega, \\
		V &= B^{\D}, &&\text{ on } \Gamma_{\D} \times (0,T), \\
		\dz V &= B^{\mathrm{N}}, &&\text{ on } \Gamma_{\mathrm{N}} \times (0,T),
	\end{aligned}
	\right. 
	\label{eq:heat2}
\end{equation}
where $B v = \frac{\dz v(b)-\dz v(a)}{b-a}$.
The next lemma addresses the associated optimal $\rL^p$-$\rL^q$-regularity result.

\begin{lem}\label{lem:heat equation}
Let $p,q,\mu$ be as in \autoref{thm:optdata}. Then there exists $\lambda \ge 0$ such that system \eqref{eq:heat2} has a unique solution $V \in \E_{1,\mu}$
if and only if the data $(H,V_0,B^{\D},B^{\rN})$ fulfill $H \in \E_{0,\mu}$, $V_0 \in \rX_{\gamma,\mu}$, $B^{\D} \in \F_{\D,\mu}$, $B^{\rN} \in \F_{\rN,\mu}$ as well as 
$1 - \frac{1}{2q} > \frac{1}{p} + 1 - \mu$ and $\dz V_0 = B^{\rN}(0)$ if $\frac{1}{2} - \frac{1}{2q} > \frac{1}{p} + 1 - \mu$.
Furthermore, there exists a constant $C>0$ (depending only on $T>0$, $p$, $q$ and $\mu$) such that  
\begin{equation*}
	\| V \|_{\E_{1,\mu}} \le C \cdot \bigl( \| H \|_{\E_{0,\mu}} + \| v_0 \|_{X_{\gamma,\mu}} + \| B^{\D} \|_{\F_{\D,\mu}} + \| B^{\mathrm{N}} \|_{\F_{\rN,\mu}}\bigr).
\end{equation*}
Hence, the solution operator 
\begin{equation} 
	S_V \colon \E_{0,\mu} \times \rX_{\gamma,\mu} \times \F_{\D,\mu} \times \F_{\rN,\mu} \to \E_{1,\mu}, \quad (H,V_0,B^{\D},B^{\rN}) \mapsto V \label{eq:SV}
\end{equation}  
is an isomorphism. 
\end{lem}

\begin{proof}
We start by noting that the necessity of the conditions given in Assumption 4.1 follows by standard arguments of trace theory. This is not the case  
for the sufficiency part of the proof. We divide this proof into several steps. Let us note first that time-weighted maximal $\rL^p$-regularity is equivalent 
to unweighted maximal $\rL^p$-regularity, see~\cite{PS:16}. 

\begin{step}\label{step1}
For homogeneous boundary conditions $B^{\D} = B^{\mathrm{N}} = 0$, we deduce the maximal $\rL_\mu^p$-regularity of $A V \coloneqq - \Delta V + \lambda V  + B V$ from the maximal regularity of 
the operator $-\Delta + \lambda \colon \rW^{2,q}_{\bc}(\Omega) \subset \rLq(\Omega) \to \rLq(\Omega)$ and the relative $\Delta$-boundedness of $B$, see e.\ g.~\cite{DHP:03}.
\end{step}

\begin{step}
Consider the perturbed problem
\begin{equation}
	\left\{
	\begin{aligned}
		\dt V - \Delta V+ \lambda V + s B V &= H, &&\onOmegaT, \\
		V(0) &= V_0, &&\onOmega, \\
		V &= B^{\D}, &&\text{ on } \Gamma_{\D} \times (0,T), \\
		\dz V &= B^{\mathrm{N}}, &&\text{ on } \Gamma_{\mathrm{N}} \times (0,T),
	\end{aligned}
	\right. 
	\label{eq:heat3}
\end{equation}
for $s \in [0,1]$. If $s = 0$, the optimal data result is known for standard domains by the results given in \cite{DHP:07}. It carries over to the situation of cylindrical domains 
by a localization procedure.  Hence, there exists a unique solution $V \in \E_{1,\mu}$ to \eqref{eq:heat3} if and only if the assumptions from the assertion are satisfied. 
\end{step}

\begin{step}\label{step3}
W.\ l.\ o.\ g., we may assume $V_0 = 0$. By \autoref{step1}, we have $\dt - \Delta + \lambda + s B \colon _0\E_{1,\mu} \to \E_{0,\mu}$.
On the other hand, 
$V|_{\Gamma_{\D}} \in \rLp_\mu(0,T;\rB_{qq}^{2-\nicefrac{1}{q}}(G;\R^2))$ and $\dz V|_{\Gamma_{\rN}}  \in \rLp_\mu(0,T;\rB_{qq}^{1-\nicefrac{1}{q}}(G;\R^2))$.

In the sequel, the prescript $_0$ indicates homogeneous initial values.
Arguing as in the proof of the necessity part in \cite{DHP:07}, we find that 
$V|_{\Gamma_{\D}} \in {_0\rF_{pq,\mu}^{1 - \nicefrac{1}{2q}}}(\R_+;\rLq(G;\R^2))$ as well as $\dz V|_{\Gamma_{\rN}} \in {_0\rF_{pq,\mu}^{\nicefrac{1}{2} - \nicefrac{1}{2q}}}(\R_+;\rLq(G;\R^2))$. We hence conclude that
\begin{equation*}
    (\dt - \Delta +\lambda + s B, \tr|_{\Gamma_{\D}}, \dz|_{\Gamma_{\rN}}) \colon _0\E_{1,\mu} \to \E_{0,\mu} \times \F_{\D,\mu} \times \F_{\rN,\mu}
\end{equation*}
is a bounded linear operator. The maximal regularity yields the estimate
\begin{equation*}
    \begin{aligned}
        \| V \|_{\E_{1,\mu}} 
       &\le C \bigl(\| \dt V - \Delta V + \lambda V + sBV \|_{\E_{0,\mu}} + \| BV \|_{\E_{0,\mu}} +  \| V|_{\Gamma_{\D}} \|_{\F_{\D,\mu}} + \| \dz V_{\Gamma_{\rN}} \|_{\F_{\rN,\mu}}\bigr),
    \end{aligned}
\end{equation*}
where the constant $C > 0$ is independent of $s$. We proceed with the continuity method. For this, we recall e.\ g.~from \cite{Ama:95} the embedding 
$_0\E_{1,\mu}(0,T) \hra \mathrm{BUC}([0,T];\rX_{\gamma,\mu})$, where the notation $\E_{1,\mu}(0,T)$ means that we consider the space on the time interval $(0,T)$.
We will typically omit this piece of notation in the remainder of the section unless it helps to avoid confusion.
Note that the embedding constant in the preceding embedding is independent of $T$. This leads to 
\begin{equation}\label{eq:B estimate} 
    \begin{aligned}
        \| BV \|_{\E_{0,\mu}} &\le C \cdot \| V \|_{\rLp_\mu(0,T;\rW^{1+\nicefrac{1}{q} + \eps,q}(\Omega))} \le C \cdot \eps \| V \|_{\E_{1,\mu}} + \frac{C}{\eps} \cdot \| V \|_{\rLp_\mu(0,T;\rLq(\Omega))}\\
        &\le C \cdot \eps \| V \|_{\E_{1,\mu}} + \frac{C}{\eps (p(1-\mu) + 1)^{\nicefrac{1}{p}}} \cdot T^{1-\mu + \nicefrac{1}{p}} \cdot \| V \|_{\mathrm{BUC}([0,T];\rX_{\gamma,\mu})}\\
        &\le C \cdot \bigl(\eps + \frac{T^{1-\mu + \nicefrac{1}{p}}}{\eps (p(1-\mu) + 1)^{\nicefrac{1}{p}}}\bigr) \| V \|_{\E_{1,\mu}}.
    \end{aligned}
\end{equation}
Thus, choosing $\eps > 0$ and $T > 0$ sufficiently small, we employ an absorption argument to deduce that
\begin{equation*}
    \| V \|_{\E_{1,\mu}} \le C\bigl(\left\| \dt V - \Delta V + \lambda V +sBV \right\|_{\E_{0,\mu}} +  \left\| V|_{\Gamma_{\D}} \right\|_{\F_{\D,\mu}} + \left\| \dz V_{\Gamma_{\rN}} \right\|_{\F_{\rN,\mu}}\bigr).
\end{equation*} 
The method of continuity then yields that
\begin{equation*}
    (\dt - \Delta +\lambda + s B, \tr|_{\Gamma_{\D}}, \dz|_{\Gamma_{\rN}}) \colon _0\E_{1,\mu} \to \E_{0,\mu} \times \F_{\D,\mu} \times \F_{\rN,\mu}
\end{equation*}
is an isomorphism. 
Hence, \eqref{eq:heat3} admits a solution for $s = 1$ and $T > 0$ small, with $T$ independent of $V_0$.
\end{step}

\begin{step}\label{step4}
In this final step, we extend the result from \autoref{step3} to large $T$. To this end, let $\frac{9}{10} \cdot T' < T_1 < T'$, where $T'$ is the largest $T$ from \autoref{step3}, 
and $V^1 \in \E_{1,\mu}(0,T_1) \hookrightarrow \mathrm{BUC}([0,T_1];\rX_{\gamma,\mu})$ is the solution of \eqref{eq:heat2}. 
We consider the problem \eqref{eq:heat3} on the time interval $[T_1,2T_1]$ with initial data $V(T_1) = V^1(T_1)$. 
This is legitimate in view of $V^1(T_1) \in \rX_{\gamma,\mu}$.
		Repeating the argument of \autoref{step3}, we obtain the existence of a unique solution $V^2$ of \eqref{eq:heat2} on $[T_1,2T_1]$. 
        Since $[0,T'] \cap [T_1,2T_1] \neq \emptyset$, we deduce from the uniqueness the equality~$V_1|_{[0,T'] \cap [T_1,2T_1]} = V_2|_{[0,T'] \cap [T_1,2T_1]}$.
        As a result, we get a unique solution $V \in \E_{1,\mu}(0,2T_1)$ of \eqref{eq:heat2}.
        Iterating this argument, we obtain the unique solution $V \in \E_{1,\mu}(0,T)$.
        \qedhere
\end{step}
\end{proof}

We now proceed with the proof of \autoref{thm:optdata}. It is again subdivided into several steps.

\begin{proof}[Proof of \autoref{thm:optdata}]
\setcounter{stp}{0}
\begin{step}\label{step11}
By \eqref{eq:hstokes1}$_2$ and \eqref{eq:hstokes1}$_3$, the pressure $\pi$ and the vertical velocity $w$ can be deduced as
\begin{align}
    \pi &= \pi_s + \int_a^z \dz \pi \d \xi = \pi_s + \int_a^z f_w \d \xi \quad \text{ and } \label{eq:pressure}\\
    w &= w(a) + \int_a^z \dz w \d \xi = b_{w,a}^{\D} + \int_a^z f_{\mdiv} \d \xi - \int_a^z \divH v \d \xi .
\label{eq:w} 
\end{align}
As $\divH \vbar = \frac{b_{w,a}^{\D} - b_{w,b}^{\D}}{b-a} + \fbar_{\mdiv} =: g_{\mdiv}$, \eqref{eq:hstokes1} is equivalent to
\begin{equation}
	\left\{
	\begin{aligned}
		\dt v - \Delta v + \lambda v + \nablaH \pi_s &= g_v, &&\on \Omega \times (0,T), \\
		\partial_z \pi_s &= 0, &&\on \Omega \times (0,T), \\
		\divH \vbar &= g_{\mdiv}, &&\on G \times (0,T), 
	\end{aligned}
	\right. 
	\label{eq:hstokes3}
\end{equation}
where $g_v := f_v - \nablaH \int_a^z f_w \d \xi$, and the system is complemented by initial conditions $v(0) = v_0$ on $\Omega$ as well as boundary conditions $v = b_v^{\D}$ on $\Gamma_{\D} \times (0,T)$ and $\dz v = b_v^{\mathrm{N}}$ on $\Gamma_{\mathrm{N}} \times (0,T)$.
\end{step}

\begin{step}\label{step22}
We now split into barotropic and baroclinic modes. 
Taking the vertical average in \eqref{eq:hstokes3}$_1$, we find $\dt \vbar - \DeltaH \vbar +\lambda \vbar + \nablaH \pi_s = \bar{g}_v + B \vtilde$, with $B \vtilde = \frac{\dz \vtilde(b)-\dz \vtilde(a)}{b-a}$, so $\vbar$ satisfies the 2D Stokes problem
\begin{equation}
	\left\{
	\begin{aligned}
		\dt \vbar - \DeltaH \vbar + \lambda \vbar + \nablaH \pi_s &= \bar{g}_v + B \vtilde, &&\on G \times (0,T), \\
		\divH \vbar &= g_{\mdiv}, &&\on G \times (0,T), \\
		\vbar(0) &= \vbar_0, &&\on G.
	\end{aligned}
	\right. 
	\label{eq:stokes}
\end{equation}
Subtracting \eqref{eq:stokes} from \eqref{eq:hstokes3}, we obtain the baroclinic mode
\begin{equation}
	\left\{
	\begin{aligned}
		\dt \vtilde - \Delta \vtilde + \lambda \vtilde + B \vtilde &= \tilde{g}_v, &&\on \Omega \times (0,T), \\
		\vtilde(0) &= \vtilde_0, &&\on \Omega, \\
		\vtilde &= b_v^{\D} - \vbar, &&\text{ on } \Gamma_{\D} \times (0,T), \\
		\dz \vtilde &= b_v^{\mathrm{N}}, &&\text{ on } \Gamma_{\mathrm{N}} \times (0,T).
	\end{aligned}
	\right. 
	\label{eq:heat1}
\end{equation}
\end{step}

\begin{step}{\em Sufficiency}. 
By \autoref{step11}, solving \eqref{eq:hstokes1} is equivalent to solving \eqref{eq:hstokes3}. Thanks to \autoref{step22}, it suffices to solve the coupled problems \eqref{eq:stokes} and \eqref{eq:heat1}.
Consider the map 
$$
R \colon \E_{1,\mu} \to \E_{1,\mu} \quad \mbox{ with } \quad V \mapsto S_V (\tilde{g}_v, \tilde{v}_0,b_v^{\D} - S_U ( \bar{g}_v+BV,g_{\mdiv},\vbar_0), b_v^{\rN}),
$$
where $S_U$ denotes the solution map associated to the 2D Stokes problem from \eqref{eq:stokes}, whose existence is guaranteed by \cite[Chapter~7]{PS:16}, and 
$S_V$ represents the solution map described in \eqref{eq:SV}. 
The coupled problems~\eqref{eq:stokes} and \eqref{eq:heat1} then admit a unique solution $(\vbar,\vtilde)$ if and only if $R$ has a unique fixed point.

First, we verify that $R \colon \E_{1,\mu} \to \E_{1,\mu}$ is well-defined and start with $S_U( \bar{g}_v+BV,g_{\mdiv},\vbar_0)$ for $V \in \E_{1,\mu}$.
From \autoref{assu:data}(i), we derive
$    \bar{g}_v = \fbar_v - \overline{\int_a^z \nablaH f_w \d \xi} \in \rLp(0,T ; \rLq (G;\R^2)) \eqqcolon \E_{0,\mu}^{\rS}$. 
Concatenating $V \in \E_{1,\mu}$ with $B \in \sL(\rW^{1+\nicefrac{1}{q}+\eps,q}(\Omega),\rLq(G))$, we deduce that $BV \in \E_{0,\mu}^{\rS}$, so $\bar{g}_v+BV \in \E_{0,\mu}^{\rS}$.
\autoref{assu:data}(ii) now yields that 
$    g_{\mdiv} = \fbar_{\mdiv} + \frac{b_{w,a}^{\D} - b_{w,b}^{\D}}{b-a} \in \G_{1,\mu}.$
Additionally invoking \autoref{assu:data}(iii) and (v), we find that $\vbar_0 \in \rB_{qp}^{2(\mu-\nicefrac{1}{p})}(G;\R^2) \eqqcolon \rX_{\gamma,\mu}^{\rS}$ and $\divH \vbar_0 = \frac{b_{w,a}^{\D}(0) - b_{w,b}^{\D}(0)}{b-a} + \fbar_{\mdiv}(0) \in \sD'(G)$.
As a consequence, the term~$S_U( \bar{g}_v+BV,g_{\mdiv},\vbar_0) \colon \E_{0,\mu}^{\rS} \times \G_{1,\mu} \times \rX_{\gamma,\mu}^{\rS} \to \E_{1,\mu}^{\rS}$ is well-defined by \cite[Chapter~7]{PS:16}.
For brevity, we will denote the resulting solution $S_U( \bar{g}_v+BV,g_{\mdiv},\vbar_0)$ by $U_V$.

It remains to check that $S_V (\tilde{g}_v, \tilde{v}_0,b_v^{\D} - U_V, b_v^{\rN})$ is also well-defined.
First, by virtue of \autoref{assu:data}(i), observe that
$
    \tilde{g}_v = \ftilde_v - \widetilde{\int_a^z \nablaH f_w \d \xi} \in  \E_{0,\mu}.
$
Further, \autoref{assu:data}(iii) implies $\Tilde{v}_0 \in  \rX_{\gamma,\mu}$. Next, observe the relation $U_V \in \rL_\mu^p(0,T;\rH^{2,q}(G;\R^2)) \subset \rLp_\mu(0,T;\rB_{qq}^{2-\nicefrac{1}{q}}(G;\R^2))$. 
\autoref{assu:data}(iv) yields $b_v^{\D} - U_V \in \F_{\D,\mu}$ provided $\rH_\mu^{1,p}(0,T;E) \hra \rF_{pq,\mu}^{1-\nicefrac{1}{2q}}(0,T;E)$ for some Banach space $E$. 
Indeed, we have
\begin{equation*}
    \rH_\mu^{1,p}(0,T;E) \hra \rF_{pq,\mu}^{1-\eps}(0,T;E) \hra \rF_{pq,\mu}^{1-\nicefrac{1}{2q}}(0,T;E)
\end{equation*}
for $\eps > 0$ sufficiently small.
This results in $b_v^{\D} - U_V \in \F_{\D,\mu}$.
\autoref{assu:data}(iv) in turn guarantees that~
$\vtilde_0 = b_v^{\D}(0) - \vbar_0 \text{ if } 1 - \frac{1}{2q} > \frac{1}{p} + 1 - \mu$ as well as  $\dz \vtilde_0 = b_v^{\rN} \text{ if } \frac{1}{2} - \frac{1}{2q} > \frac{1}{p} 
+ 1 - \mu$.
Thanks to \autoref{lem:heat equation}, it follows that $S_V (\tilde{g}_v, \tilde{v}_0,b_v^{\D} - U_V, b_v^{\rN}) \colon \E_{0,\mu} \times \rX_{\gamma,\mu} \times \F_{\D,\mu} \times \F_{\rN,\mu} \to \E_{1,\mu}$ is well-defined.
Putting this together with the above result on $U_V$, we argue that $R \colon \E_{1,\mu} \to \E_{1,\mu}$ is also well-defined.

Last, we prove the contraction property of $R$. 
By linearity, for $U_i := S_U( \bar{g}_v+BV_i,g_{\mdiv},\vbar_0)$, we first obtain~$U_1 - U_2 = S_U(B(V_1-V_2),0,0)$.
Upon using the embedding $\E_{1,\mu}^{\rS} \hra \F_{\D,\mu}$, \autoref{lem:heat equation} yields 
\begin{equation*}
    \| R(V_1) - R(V_2) \|_{\E_{1,\mu}} = \| S_V(0, 0,- S_U(B(V_1-V_2),0,0),0) \|_{\E_{1,\mu}} \le C \cdot \| S_U(B(V_1-V_2),0,0) \|_{\F_{\D,\mu}}. 
\end{equation*}
With regard to the estimate for $\| S_U(B(V_1-V_2),0,0) \|_{\E_{1,\mu}^{\rS}}$, we note that $S_U(B(V_1-V_2),0,0)$ solves a 2D~Stokes equation with right-hand side $F = B(V_1 - V_2)$.
Hence, the corresponding optimal data result, see \cite[Chapter~7]{PS:16}, implies that
\begin{equation*}
    \| S_U(B(V_1-V_2),0,0) \|_{\E_{1,\mu}^{\rS}} \le C \cdot \| B(V_1 - V_2) \|_{\E_{0,\mu}^{\rS}}.
\end{equation*}
Combining this estimate  with \eqref{eq:B estimate}, for a $T$-independent constant $C > 0$, we derive 
\begin{equation*}
     \| R(V_1) - R(V_2) \|_{\E_{1,\mu}} \le C \cdot \bigl(\eps + \frac{T^{1-\mu + \nicefrac{1}{p}}}{\eps (p(1-\mu) + 1)^{\nicefrac{1}{p}}}\bigr) \| V_1  -V_2 \|_{\E_{1,\mu}} \le \kappa \| V_1  -V_2 \|_{\E_{1,\mu}},
\end{equation*}
with $\kappa \in (0,1)$ provided $\eps > 0$ and $T > 0$ are sufficiently small.
The contraction mapping principle then yields the existence of a unique $V_* \in \E_{1,\mu}$ such that $R V_* = V_*$, i.\ e., a unique solution to \eqref{eq:heat1} and then also to \eqref{eq:stokes}.
The extension to arbitrary $T$ can be obtained as in the   in the proof of \autoref{lem:heat equation}.
\end{step}

\begin{step}{\em Necessity}.
Let $v \in \E_{1,\mu} $ be the unique strong solution of \eqref{eq:hstokes1}. We then obtain a unique,  strong solutions $\vbar \in \E_{1,\mu}^{\rS}$ of \eqref{eq:stokes} and 
$\vtilde \in \E_{1,\mu}$ of \eqref{eq:heat1}. Furthermore, $\bar{g}_v \in \E_{0,\mu}$, $g_{\mdiv} \in \G_{1,\mu}$, $\vbar_0 \in \rX_{\gamma,\mu}^{\rS}$ and $\divH \vbar_0 = g_{\mdiv}(0)$ in $\sD'(G)$.
Thus  \autoref{assu:data}(i), (ii), (iii) and (v) are satisfied. 

For $\vtilde$, we invoke $\E_{1,\mu}^{\rS} \hra \F_{\D,\mu}$ and use \autoref{lem:heat equation} to infer that $\Tilde{g}_v \in \E_{0,\mu}$, $\vtilde_0 \in \rX_{\gamma,\mu}$, $b_v^{\D} \in \F_{\D,\mu}$, $b_v^{\rN} \in \F_{\rN,\mu}$ and $v_0 = b_v^{\D}(0)$ if $1 - \frac{1}{2q} > \frac{1}{p} + 1 - \mu$ as well as $\dz v_0 = b_v^{\rN}(0)$ in the case that $\frac{1}{2} - \frac{1}{2q} > \frac{1}{p} + 1 - \mu$.
Hence, \autoref{assu:data} is valid. \qedhere
\end{step}
\end{proof}

Finally, let us note that it is also possible to handle the situation with non-constant coefficients as for the equations accounting for the atmosphere in \eqref{eq:CAO}.

\begin{rem}\label{rem:necessary adjustment for atmospheric eq}{\rm
It is also possible to consider instead of  \eqref{eq:hstokes1}$_1$ the equation 
\begin{equation*}
    \dt v - \DeltaH v - c \partial_p(p^2 \partial_p v) + \nablaH \pi = f_v, \text{ on } \Omegaatm \times (0,T),
\end{equation*}
where $c > 0$ is a constant.
In fact, taking the vertical average in this equation, we find that the relevant perturbative operator $B$ can still be used.
Hence, it remains to argue that an analogue of \autoref{lem:heat equation} still holds true for the resulting adjusted version of \eqref{eq:heat2}.
On the other hand, we calculate
\begin{equation*}
    c \partial_p(p^2 \partial_p v) = c p^2 \partial_{pp} v + 2 c p \partial_p.
\end{equation*}
Thanks to the parameter-ellipticity of the resulting non-divergence form operator $\DeltaH v + c p^2 \partial_{pp} v + 2 c p \partial_p v$ with smooth and 
bounded coefficients, the assertion of \autoref{thm:optdata} remains valid also in this case. 
}
\end{rem}

\section{Local well-posedness and blow-up criterion}\label{sec:local+blowup}

In this section, we address first the local well-posedness of the CAO-system. 
In addition, we provide a blow-up criterion for the coupled system that will be used to derive the global well-posedness by means of a-priori estimates in \autoref{sec:global}.

The proof of the local well-posedness is based on the optimal data theory for the primitive equations presented in \autoref{sec:linearizedproblem}.
We emphasize that sharp estimates for the boundary terms in  Triebel-Lizorkin spaces guarantee to handle the boundary terms in the fixed point argument and in the blow-up criterion.

We first introduce some notation. 
The ground spaces are given by $\rX_0^\atm \coloneqq \rL_{\sigmabar}^q(\Omegaatm)$ and~$\rX_0^\ocn \coloneqq \rL_{\sigmabar}^q(\Omegaocn)$.
Furthermore, for $\beta \in (0,1)$, we define $\rX_\beta^\atm \coloneqq \rH^{2 \beta,q}(\Omegaatm;\R^2) \cap \rX_0^{\air}$ as well as $\rX_\beta^\ocn \coloneqq \rH^{2 \beta,q}(\Omegaocn;\R^2) \cap \rX_0^{\ocn}$ and set~$\rX_1^{\atm} \coloneqq \rH^{2,q}(\Omegaatm;\R^2) \cap \rX_0^{\air}$ and $\rX_1^{\ocn} \coloneqq \rH^{2,q}(\Omegaocn;\R^2) \cap \rX_0^{\ocn}$.
For $\mu \in (\nicefrac{1}{p},1]$, the data spaces are defined by
\begin{equation*}
    \E_{0,\mu,\air} \coloneqq \rL_\mu^p(0,T;\rX_0^\atm), \enspace \E_{0,\mu,\ocn} \coloneqq \rL_\mu^p(0,T;\rX_0^\ocn) \enspace \text{as well as} \enspace \E_{0,\mu,T} \coloneqq \E_{0,\mu,\air} \times \E_{0,\mu,\ocn}.
\end{equation*}
Moreover, the solution spaces are given by 
\begin{equation*}
    \begin{aligned}
        \E_{1,\mu,\atm} 
        \coloneqq \rH_\mu^{1,p}(0,T;\rX_0^\atm) \cap \rL_\mu^{p}(0,T;\rX_1^\atm), \, \E_{1,\mu,\ocn} 
        \coloneqq \rH_\mu^{1,p}(0,T;\rX_0^\ocn) \cap \rL_\mu^{p}(0,T;\rX_1^\ocn) \enspace  \text {and} \enspace \E_{1,\mu,T} \coloneqq \E_{1,\mu,\atm} \times \E_{1,\mu,\ocn}.
    \end{aligned}
\end{equation*}

The spaces for the boundary data in the present case of Neumann boundary conditions are given by
\begin{equation*}
    \begin{aligned}
        \F_{\mu,\air} 
        &\coloneqq \rF_{pq,\mu}^{\nicefrac{1}{2}-\nicefrac{1}{2q}}(0,T;\rLq(G;\R^2)) \cap \rLp_\mu(0,T;\rB_{qq}^{1 - \nicefrac{1}{q}}(G;\R^2)),\\
        \F_{\mu,\ocn} 
        &\coloneqq \rF_{pq,\mu}^{\nicefrac{1}{2}-\nicefrac{1}{2q}}(0,T;\rLq(G;\R^2)) \cap \rLp_\mu(0,T;\rB_{qq}^{1 - \nicefrac{1}{q}}(G;\R^2)) \enspace \text{and} \enspace \F_{\mu,T} \coloneqq \F_{\mu,\air} \times \F_{\mu,\ocn}.
    \end{aligned}
\end{equation*}

\subsection{Local well-posedness}\label{ssec:local wp}
\

Throughout this section, we assume that $p_s \equiv 1$, as it does not affect the  analysis.  Our  local well-posedness result then reads as follows.

\begin{prop}\label{prop:localwellposed}
Let $T>0$, $p,q \in (1,\infty)$ satisfy $\mu_c = \frac{1}{p} + \frac{1}{q} \leq 1$ and $\mu = \mu_c$ if $q \ge 2$, while $\mu = \mu_c + \eps \le 1$, for small $\eps > 0$, if $q < 2$.
Let  $(\fair, \focn) \in \E_{0,\mu,T}$. 
In addition, assume $v_0 = (\vair_0, \vocn_0) \in \rB_{qp,\sigmabar}^{\nicefrac{2}{q}}(\Omegaatm) \times \rB_{qp,\sigmabar}^{\nicefrac{2}{q}}(\Omegaocn)$ if $q \ge 2$, and $v_0 \in \rB_{qp,\sigmabar}^{\nicefrac{2}{q} + \eps}(\Omegaatm) \times \rB_{qp,\sigmabar}^{\nicefrac{2}{q} + \eps}(\Omegaocn)$ if $q < 2$. 

Then there exists $0 < T' = T'(v_0) \le T$ such that the system \eqref{eq:CAO} admits a unique, strong solution~$(\vair,\omegaair,\Phi,\vocn,\wocn,\pi)$ satisfying 
 $(\vair,\vocn) \in \E_{1,\mu,T'}$.
\end{prop}

We remark that \autoref{prop:localwellposed} also holds for general time weights $\mu \in [\nicefrac{1}{p} + \nicefrac{1}{q},1]$.
The proof of \autoref{prop:localwellposed} is based on  a paraproduct estimate in weighted Triebel-Lizorkin spaces, essentially due to Chae \cite{Chae:02}.
Let us observe that it follows by adjusting  the proof of \cite[Proposition~1.2]{Chae:02} to the weighted situation by invoking the continuity of the Hardy-Littlewood maximal operator on 
weighted $\rL^q$-spaces, see \cite[Theorem~V.3.1]{Stein:93}. For related results concerning bilinear estimates for the Navier-Stokes equations in homogenous Triebel-Lizorkin spaces, we refer to 
\cite{KS:04}.

\begin{lem}
\label{lem:weighted paraprod estimate}
Let $G \in \{\R^2,\T^2\}$, $s > 0$, $p \in (1,\infty)$, $q \in (1,\infty]$, $\mu \in (\nicefrac{1}{p},1]$ and $p_1, p_2, r_1, r_2 \in (1,\infty)$ satisfying $\frac{1}{p} = \frac{1}{p_1} + \frac{1}{p_2} = \frac{1}{r_1} + \frac{1}{r_2}$.
Then there exists a constant $C > 0$ such that for $\tau_i \in (\nicefrac{1}{p_i},1]$ as well as~$\sigma_i \in (\nicefrac{1}{r_i},1]$, $i=1,2$, with $p(1-\mu) = p_i(1-\tau_i) = r_i(1-\sigma_i)$, it holds that
\begin{equation*}
    \| f g \|_{\rF_{pq,\mu}^s(G)} \le C\bigl(\| f \|_{\rL_{\tau_1}^{p_1}(G)} \| g \|_{\rF_{p_2 q,\tau_2}^s(G)} + \| g \|_{\rL_{\sigma_1}^{r_1}(G)} \| f \|_{\rF_{r_2 q,\sigma_2}^s(G)}\bigr).
\end{equation*}
\end{lem}

We now estimate the nonlinear boundary terms using the above paraproduct estimate.
Note that the estimates are essential
with regard to the local well-posedness result \autoref{prop:localwellposed} and the underlying optimal data result \autoref{thm:optdata}. 

\begin{lem}\label{lem:est of boundary term}
Let $p,q$ and $\mu$ be as in \autoref{prop:localwellposed} and $T>0$. 
Given $(\vair,\vocn) \in \E_{1,\mu,T}$, consider the boundary term $B(\vair,\vocn) \coloneqq (\vair - \vocn) |\vocn - \vair|$.
Then for $\tau$ with $1-\mu = 2(1-\tau)$, there is $C>0$ with
\begin{equation}\label{eq:est Lp Besov}
    \begin{aligned}
        \| B(\vair,\vocn) \|_{\rLp_\mu(0,T;\rB_{qq}^{1 - \nicefrac{1}{q}}(G))}
        &\le C\bigl(\| \vair \|_{\rL_\tau^{2p}(0,T;\rX_\beta^\air)}^2 + \| \vocn \|_{\rL_\tau^{2p}(0,T;\rX_\beta^\ocn)}^2\bigr)
        \le C\bigl(\| \vair \|_{\E_{1,\mu,\air}}^2 + \| \vocn \|_{\E_{1,\mu,\ocn}}^2\bigr),        
    \end{aligned}
\end{equation}
\begin{equation}\label{eq:est Triebel-Lizorkin Lq}
    \begin{aligned}
        \| B(\vair,\vocn) \|_{\rF_{pq,\mu}^{\nicefrac{1}{2}-\nicefrac{1}{2q}}(0,T;\rLq(G))}
        &\le C\bigl(\| \vair \|_{\rF_{2p q,\tau}^{\nicefrac{1}{2} - \nicefrac{1}{2q}}(0,T;\rW^{\nicefrac{1}{2q},2q}(\Omegaair))} + \| \vocn \|_{\rF_{2p q,\tau}^{\nicefrac{1}{2} - \nicefrac{1}{2q}}(0,T;\rW^{\nicefrac{1}{2q},2q}(\Omegaocn))}\bigr)\\
        &\quad \cdot \bigl(\| \vair \|_{\rL_\tau^{2p}(0,T;\rW^{\nicefrac{1}{2q},2q}(\Omegaair))} + \| \vocn \|_{\rL_\tau^{2p}(0,T;\rW^{\nicefrac{1}{2q},2q}(\Omegaocn))}\bigr) \enspace \text{and}
    \end{aligned}
\end{equation}
\begin{equation}\label{eq:total est boundary term}
    \begin{aligned}
        \| B(\vair,\vocn) \|_{\F_{\mu,\ocn}} \le C \bigl(\| \vair \|_{\E_{1,\mu,\air}}^2 + \| \vocn \|_{\E_{1,\mu,\ocn}}^2\bigr).
    \end{aligned}
\end{equation}
\end{lem}

\begin{proof}[Proof of \autoref{lem:est of boundary term}]
Recall the notation $V \coloneqq \vair|_{\Gammaairb} - \vocn|_{\Gau}$ as introduced in \eqref{eq:V}.
Hence, $B(\vair,\vocn)$ can be rewritten as $B(V) \coloneqq B(\vair,\vocn) = V |V|$.
As the embeddings for the atmospheric and the oceanic terms work similarly, we will omit the respective sub- and superscripts if a distinction is not strictly necessary.

We start with \eqref{eq:est Lp Besov} and focus on the product in the spatial norm. By \autoref{lem:weighted paraprod estimate} for $\mu = 1$, we obtain
\begin{equation}
    \begin{aligned}\label{eq:estimateBbesov}
        \| V \cdot |V| \|_{\rB_{qq}^{1-\nicefrac{1}{q}}(G)} 
        &\le \| V \cdot |V| \|_{\rF_{q2}^{1-\nicefrac{1}{q} + \eps}(G)} 
        \le C \| V \|_{\rL^{rq}(G)} \| V \|_{\rH^{1-\nicefrac{1}{q} + \eps,r'q}(G)}\\
        &\le C\bigl(\| \vair \|_{\rW^{\nicefrac{1}{rq},rq}(\Omegaair)} + \| \vocn \|_{\rW^{\nicefrac{1}{rq},rq}(\Omegaocn)}\bigr) \\
        &\quad \cdot \bigl(\| \vair \|_{\rW^{1-\nicefrac{1}{q}+2\eps + \nicefrac{1}{r'q},r'q}(\Omegaair)} + \| \vocn \|_{\rW^{1-\nicefrac{1}{q}+2\eps + \nicefrac{1}{r'q},r'q}(\Omegaocn)}\bigr),
    \end{aligned}
\end{equation}
where $\frac{1}{r} + \frac{1}{r'} = 1$.
Note that $\rX_{\beta} \hra \rH^{1+\nicefrac{1}{q},q}(\Omega;\R^2)$. Thus, there  exists $r' \in (\nicefrac{2}{q},2)$ so that the 
embedding~$\rX_\beta \hra \rW^{\nicefrac{1}{rq},rq}(\Omega;\R^2) \cap \rW^{1-\nicefrac{1}{q}+2\eps + \nicefrac{1}{r'q},r'q}(\Omega;\R^2)$ is valid
for $\eps > 0$ sufficiently small.
On the other hand, for $1 - \mu = 2(1-\tau)$, we estimate $\| f \cdot g \|_{\rLp_\mu(0,T)} \le \| f \|_{\rL_{\tau}^{2p}(0,T)} \| g \|_{\rL_{\tau}^{2p}(0,T)}$.
Finally, the mixed derivative theorem implies $\E_{1,\mu}(0,T) \hra \rL_\tau^{2p}(0,T;\rX_\beta)$, yielding \eqref{eq:est Lp Besov}. 

Using \autoref{lem:weighted paraprod estimate} we obtain
\begin{equation}\label{eq:estimatebtriebellizorkin}
    \begin{aligned}
        \| B(V) \|_{\rF_{pq,\mu}^{\nicefrac{1}{2}-\nicefrac{1}{2q}}(0,T;\rLq(G))}
        &\le C\bigl(\| \vair \|_{\rF_{2p q,\tau}^{\nicefrac{1}{2} - \nicefrac{1}{2q}}(0,T;\rW^{\nicefrac{1}{2q},2q}(\Omegaair))} + \| \vocn \|_{\rF_{2p q,\tau}^{\nicefrac{1}{2} - \nicefrac{1}{2q}}(0,T;\rW^{\nicefrac{1}{2q},2q}(\Omegaocn))}\bigr)\\
        &\quad \cdot \bigl(\| \vair \|_{\rL_\tau^{2p}(0,T;\rW^{\nicefrac{1}{2q},2q}(\Omegaair))} + \| \vocn \|_{\rL_\tau^{2p}(0,T;\rW^{\nicefrac{1}{2q},2q}(\Omegaocn))}\bigr),
    \end{aligned}
\end{equation}
where $1-\mu = 2(1-\tau)$. This yields the estimate \eqref{eq:est Triebel-Lizorkin Lq}.
 
To conclude \eqref{eq:total est boundary term}, note first that
\begin{equation}\label{eq:rel emb Bessel pot Triebel-Lizorkin}
    \rH_\mu^{s,p}(0,T;\rH^{\theta,q}(\Omega)) \hra \rF_{pq,\mu}^s(0,T;\rH^{\theta,q}(\Omega))
\end{equation}
is valid for $s, \theta \in \R$, $p \in (1,\infty)$ and $q \in [2,\infty)$.
A similar reverse embedding holds true for $q \in (1,2]$.
Moreover, observe that  $\rX_\beta \hra \rW^{\nicefrac{1}{2q},2q}(\Omega;\R^2)$, so $\E_{1,\mu} \hra \rL_\tau^{2p}(0,T;\rW^{\nicefrac{1}{2q},2q}(\Omega;\R^2))$. Concerning the first factor in \eqref{eq:estimatebtriebellizorkin}, we start with $q \ge 2$ and deduce from the mixed derivative theorem and \eqref{eq:rel emb Bessel pot Triebel-Lizorkin} that
\begin{equation}\label{eq:emb max reg space into Triebel-Lizorkin}
    \E_{1,\mu} \hra \rH_\mu^{\theta,p}(0,T;\rH^{2(1-\theta),q}(\Omega;\R^2)) \hra \rF_{pq,\mu}^{\theta}(0,T;\rW^{2(1-\theta),q}(\Omega;\R^2)).
\end{equation}
As we require that $\rW^{2(1-\theta),q}(\Omega) \hra \rW^{\nicefrac{1}{2q},2q}(\Omega)$, we demand that $\theta \le 1 - \frac{1}{q}$.
On the other hand, we exploit~\cite[Theorem~1.2]{MV:12} to get the embedding $\rF_{pq,\mu}^{\theta}(0,T;\rX) \hra \rF_{2pq,\tau}^{\nicefrac{1}{2} - \nicefrac{1}{2q}}(0,T;\rX)$.
The assumption on the critical weight $\mu_c = \frac{1}{p} + \frac{1}{q} \le 1$ guarantees the existence of such $\theta \in (0,1)$.
If $q < 2$, for $\eps > 0$ sufficiently small, we get $\rH_\mu^{\theta,p}(0,T;\rH^{2(1-\theta),q}(\Omega)) \hra \rF_{p2}^{\theta - \eps}(0,T;\rH^{2(1-\theta),q}(\Omega))$.
In conjunction with the mixed derivative theorem, this leads to
\begin{equation*}
    \E_{1,\mu} \hra \rF_{p2,\mu}^{\theta-\eps}(0,T;\rW^{2(1-\theta)- \eps,q}(\Omega;\R^2)) \hra \rF_{2pq,\tau}^{\nicefrac{1}{2} - \nicefrac{1}{2q}}(0,T;\rW^{\nicefrac{1}{2q},2q}(\Omega;\R^2)),
\end{equation*}
implying \eqref{eq:total est boundary term}.
\end{proof}

\begin{proof}[Proof of \autoref{prop:localwellposed}]
We start with a  reference solution $(\vair_*,\omegaair_*,\Phi_*,\vocn_*,\wocn_*,\pi_*)$ to \eqref{eq:CAO}, with right-hand sides $\fair$ and $\focn$ as well as with initial values 
$(\vair_*,\vocn_*)|_{t=0} =(\vair_0, \vocn_0)$.
The resulting system is decoupled, and the existence of the reference solution is guaranteed in view of the maximal regularity of the hydrostatic Stokes operator, see \cite{GGHHK:17}.
With the prescript $_0$ indicating homogeneous initial values again, we set
\begin{equation*}
    \mathbb{B}_r \coloneqq \{(\vair,\vocn) \in \E_{1,\mu,T} : (\vair,\vocn) - (\vair_*,\vocn_*) \in \prescript{}{0}{\E_{1,\mu,T}} \enspace \text{and} \enspace \| (\vair,\vocn) - (\vair_*,\vocn_*) \|_{\E_{1,\mu,T}} \le r\}.
\end{equation*}
For $(\vair_1,\vocn_1) \in \mathbb{B}_r$, we denote by $(\vair,\vocn) \coloneqq \Psi(\vair_1,\vocn_1) \in \E_{1,\mu,T}$ the solution to the linearized system \eqref{eq:CAO}, or, 
equivalently, to \eqref{eq:CAOsimplified}, where the respective right-hand sides for the primitive equations are given by
\begin{equation*}
    \begin{aligned}
        \fair + G^\air(\vair_1,\vair_1) \enspace \text{and} \enspace \focn + G^\ocn(\vocn_1,\vocn_1), \enspace \text{with} \enspace
        G^\air(\vair_1,\vair_2) = -(\vair_1 \cdot \nablaH \vair_2 + \omegaair(\vair_1) \cdot \partial_p \vair_2),
    \end{aligned}
\end{equation*}
and $G^\ocn(\vocn_1,\vocn_2)$ takes an analogous shape. Moreover, using the notational simplification $p_s \equiv 1$, the interface  conditions denoted by $\pm B(\vair_1,\vocn_1)$ are given by
\begin{equation*}
    \partial_p \vair = (\vocn_1 - \vair_1) |\vocn_1 - \vair_1| \enspace \text{and} \enspace \partial_z \vocn = (\vair_1 - \vocn_1) |\vair_1 - \vocn_1|.
\end{equation*}
The assumptions on $(\fair,\focn) \in \E_{0,\mu,T}$ and $(\vair_0,\vocn_0)$
together with the continuity of $G^\air \colon \rX_\beta^\air \times \rX_\beta^\air \to \rLq_{\sigmabar}(\Omegaair)$ and $G^\ocn \colon \rX_\beta^\ocn \times \rX_\beta^\ocn \to 
\rLq_{\sigmabar}(\Omegaocn)$, see \cite[Section~6]{GGHHK:20}, as well as \eqref{eq:total est boundary term} yield that $(\vair,\vocn) = \Psi(\vair_1,\vocn_1)$ is well-defined for 
$(\vair_1,\vocn_1) \in \mathbb{B}_r$. 

We observe next that $(\vair - \vair_*,\vocn - \vocn_*)$ solves \eqref{eq:CAO}, or, equivalently, \eqref{eq:CAOsimplified}, with terms on the right-hand side 
$G^\air(\vair_1,\vair_1)$ and $G^\ocn(\vocn_1,\vocn_1)$, boundary conditions $-B(\vair_1,\vocn_1)$ and $B(\vair_1,\vocn_1)$, and homogeneous initial values, i.\ e., $(\vatm -\vatm_\ast , \vocn - \vocn_\ast)|_{t=0} = 0$. 

It remains to show that $(\vair,\vocn) = \Psi(\vair_1,\vocn_1) \in \mathbb{B}_r$, and that $\Psi$ is a contraction. Since $(\vair,\vocn) \in \E_{1,\mu,T}$ yields 
$(\vair - \vair_*,\vocn - \vocn_*) \in \prescript{}{0}{\E_{1,\mu,T}}$, we only need to show $\| (\vair - \vair_*,\vocn - \vocn_*) \|_{\E_{1,\mu,T}} \le r$.
\autoref{thm:optdata} implies 
\begin{equation}\label{eq:optimal data est}
    \begin{aligned}
        &\quad \| (\vair - \vair_*,\vocn - \vocn_*) \|_{\E_{1,\mu,T}}\\
        &\le C\bigl(\| G^\air(\vair_1,\vair_1) \|_{\E_{0,\mu,\air}} + \| -B(\vair_1,\vocn_1) \|_{\F_{\mu,\air}}
         + \| G^\ocn(\vocn_1,\vocn_1) \|_{\E_{0,\mu,\ocn}} + \| B(\vair_1,\vocn_1) \|_{\F_{\mu,\ocn}}\bigr).
    \end{aligned}
\end{equation}
For simplicity of notation, we omit the respective sub- and superscripts in the sequel. We then find that
\begin{equation}\label{eq:est of bilinearity}
    \| G(v_1,v_1) \|_{\E_{0,\mu}} \le C(\| v_* \|_{\rL_\tau^{2p}(0,T;\rX_\beta)}^2 + \| v_1 - v_* \|_{\E_{1,\mu}}^2) \le C(\| v_* \|_{\rL_\tau^{2p}(0,T ;\rX_\beta)}^2 + r^2),
\end{equation}
where the constant $C > 0$ is time-independent. Observing the relation $v_* \in \rL_\tau^{2p}(0,T;\rX_\beta)$ for the reference solution $v_*$, we can make the corresponding term in \eqref{eq:est of bilinearity} arbitrarily 
small by choosing $T>0$ small enough.

Concerning the estimate of the boundary term, we deduce first from \eqref{eq:est Lp Besov} that
\begin{equation}\label{eq:est of first part of boundary norm}
    \begin{aligned}
        \| B(\vair_1,\vocn_1)& \|_{\rLp_\mu(0,T;\rB_{qq}^{1 - \nicefrac{1}{q}}(G))}
        \le C\bigl(\| \vair_1 \|_{\rL_\tau^{2p}(0,T;\rX_\beta^\air)}^2 + \| \vocn_1 \|_{\rL_\tau^{2p}(0,T;\rX_\beta^\ocn)}^2\bigr)\\
        &\le C\bigl(\| \vair_1 - \vair_* \|_{\rL_\tau^{2p}(0,T;\rX_\beta^\air)}^2 + \| \vair_* \|_{\rL_\tau^{2p}(0,T;\rX_\beta^\air)}^2
         + \| \vocn_1 - \vocn_* \|_{\rL_\tau^{2p}(0,T;\rX_\beta^\ocn)}^2 + \| \vocn_* \|_{\rL_\tau^{2p}(0,T;\rX_\beta^\ocn)}^2\bigr)\\
        &\le C\bigl(r^2 + \| \vair_* \|_{\rL_\tau^{2p}(0,T;\rX_\beta^\air)}^2 + \| \vocn_* \|_{\rL_\tau^{2p}(0,T;\rX_\beta^\ocn)}^2\bigr),
    \end{aligned}
\end{equation}
where $C>0$ is independent of time.
The second and the third term on the right-hand side of \eqref{eq:est of first part of boundary norm} decrease to zero as $T \to 0$. Furthermore, \eqref{eq:est Triebel-Lizorkin Lq} implies that
\begin{equation*}
    \begin{aligned}
        \| B(\vair_1,\vocn_1) \|_{\rF_{pq,\mu}^{\nicefrac{1}{2}-\nicefrac{1}{2q}}(0,T;\rLq(G))}
        &\le C\bigl(\| \vair_1 \|_{\rF_{2p q,\tau}^{\nicefrac{1}{2} - \nicefrac{1}{2q}}(0,T;\rW^{\nicefrac{1}{2q},2q}(\Omegaair))} + \| \vocn_1 \|_{\rF_{2p q,\tau}^{\nicefrac{1}{2} - \nicefrac{1}{2q}}(0,T;\rW^{\nicefrac{1}{2q},2q}(\Omegaocn))}\bigr)\\
        &\quad \cdot \bigl(\| \vair_1 \|_{\rL_\tau^{2p}(0,T;\rW^{\nicefrac{1}{2q},2q}(\Omegaair))} + \| \vocn_1 \|_{\rL_\tau^{2p}(0,T;\rW^{\nicefrac{1}{2q},2q}(\Omegaocn))}\bigr).
    \end{aligned}
\end{equation*}
Similarly as in \eqref{eq:est of first part of boundary norm}, we see that 
\begin{equation}\label{eq:est of second factor of second part of boundary norm}
    \begin{aligned}
        \| \vair_1 \|_{\rL_\tau^{2p}(0,T;\rW^{\nicefrac{1}{2q},2q}(\Omegaair))} + \| \vocn_1 \|_{\rL_\tau^{2p}(0,T;\rW^{\nicefrac{1}{2q},2q}(\Omegaocn))}
        &\le C\bigl(r + \| \vair_* \|_{\rL_\tau^{2p}(0,T;\rX_\beta^\air)} + \| \vocn_* \|_{\rL_\tau^{2p}(0,T;\rX_\beta^\ocn)}\bigr).
    \end{aligned}
\end{equation}
We now estimate as in the proof of \autoref{lem:est of boundary term}
\begin{equation}\label{eq:est of first factor of second part of boundary norm}
    \begin{aligned}
         &\quad \| \vair_1 \|_{\rF_{2p q,\tau}^{\nicefrac{1}{2} - \nicefrac{1}{2q}}(0,T;\rW^{\nicefrac{1}{2q},2q}(\Omegaair))}  + 
        \| \vocn_1 \|_{\rF_{2p q,\tau}^{\nicefrac{1}{2} - \nicefrac{1}{2q}}(0,T;\rW^{\nicefrac{1}{2q},2q}(\Omegaocn))}\\
        &\le \| \vair_1 - \vair_* \ \|_{\rF_{2p q,\tau}^{\nicefrac{1}{2} - \nicefrac{1}{2q}}(0,T;\rW^{\nicefrac{1}{2q},2q}(\Omegaair))} + \| \vair_* \ \|_{\rF_{2p q,\tau}^{\nicefrac{1}{2} - \nicefrac{1}{2q}}(0,T;\rW^{\nicefrac{1}{2q},2q}(\Omegaair))}\\
        &\quad + \| \vocn_1 - \vocn_* \ \|_{\rF_{2p q,\tau}^{\nicefrac{1}{2} - \nicefrac{1}{2q}}(0,T;\rW^{\nicefrac{1}{2q},2q}(\Omegaocn))} + \| \vocn_* \ \|_{\rF_{2p q,\tau}^{\nicefrac{1}{2} - \nicefrac{1}{2q}}(0,T;\rW^{\nicefrac{1}{2q},2q}(\Omegaocn))}\\
        &\le C\bigl(\| \vair_1 - \vair_* \|_{\E_{1,\mu,\air}} + \| \vair_* \ \|_{\rF_{2p q,\tau}^{\nicefrac{1}{2} - \nicefrac{1}{2q}}(0,T;\rW^{\nicefrac{1}{2q},2q}(\Omegaair))}\\
        &\quad + \| \vocn_1 - \vocn_* \|_{\E_{1,\mu,\ocn}} + \| \vocn_* \ \|_{\rF_{2p q,\tau}^{\nicefrac{1}{2} - \nicefrac{1}{2q}}(0,T;\rW^{\nicefrac{1}{2q},2q}(\Omegaocn))}\bigr)\\
        &\le C\bigl(r + \| \vair_* \ \|_{\rF_{2p q,\tau}^{\nicefrac{1}{2} - \nicefrac{1}{2q}}(0,T;\rW^{\nicefrac{1}{2q},2q}(\Omegaair))} + \| \vocn_* \ \|_{\rF_{2p q,\tau}^{\nicefrac{1}{2} - \nicefrac{1}{2q}}(0,T;\rW^{\nicefrac{1}{2q},2q}(\Omegaocn))}\bigr)
    \end{aligned}
\end{equation}
for $C>0$ independent of $T$.
By \eqref{eq:emb max reg space into Triebel-Lizorkin}, we have $\vair_* \in \rF_{2p q,\tau}^{\nicefrac{1}{2} - \nicefrac{1}{2q}}(0,T;\rW^{\nicefrac{1}{2q},2q}(\Omegaair))$ and likewise for $\vocn_*$.
Combining \eqref{eq:est of bilinearity}, \eqref{eq:est of first part of boundary norm}, \eqref{eq:est of second factor of second part of boundary norm} and \eqref{eq:est of first factor of second part of boundary norm}, and choosing $r>0$ and $T>0$ sufficiently small, we conclude from \eqref{eq:optimal data est} that $\Psi$ is a self-map, i.\ e., $\Psi \colon \mathbb{B}_r \to \mathbb{B}_r$. The contraction property follows in the same way.
\end{proof}

\subsection{Blow-up criterion}\label{ssec:blow-up}
\
We now formulate a blow-up criterion which will be of central importance for the proof of global existence.
		
\begin{prop}\label{thm:blowup}
Let $p,q \in (1,\infty)$, $\mu_c = \frac{1}{p} + \frac{1}{q} \leq 1$, assume  $2 \le q \le p$, and take into account initial data~$v_0 = (\vair_0,\vocn_0) \in \rB_{qp,\sigmabar}^{\nicefrac{2}{q}}(\Omegaatm) \times \rB_{qp,\sigmabar}^{\nicefrac{2}{q}}(\Omegaocn)$. Let 
$(\vair,\vocn) \in \E_{1,\mu,T}$ be the unique, local strong solution to~\eqref{eq:CAO} obtained by \autoref{prop:localwellposed}.

Then the solution $(\vair,\vocn)$ exists on a maximal time interval $[0,t_+)$, and $t_+ < \infty$ implies that $(\vair,\vocn)$ satisfies $(\vair,\vocn) \notin \rL^p((0,t_+);\rX_{\mu_c}^\air \times \rX_{\mu_c}^\ocn)$.
\end{prop}
		
\begin{proof}
To simplify the notation, we again omit the sub- and superscripts.
The mixed derivative theorem and Sobolev embeddings yield
\begin{equation}
	\begin{aligned}
		(\rLp(0,T; \rX_{\mu_c}),\E_{1,\mu_c})_{\nicefrac{1}{2}} 
        \hra \rH^{\nicefrac{(1-\alpha)}{2},p}_{\tau}(0,T ; \rH^{\mu_c + \alpha,q}(\Omega;\R^2)) \hookrightarrow \rL^{2p}_\tau (0,T ; \rH^{\mu_c + 1 -\nicefrac{1}{p},q}(\Omega;\R^2)),
	\end{aligned}
\label{eq:emeddingmaxreg}
\end{equation}
where $\alpha = 1 -\frac{1}{p}$ and $\tau = \frac{1+\mu_c}{2}$.
Now, suppose $t_+ < \infty$, and let $T_0 \in (0,t_+)$ be fixed. \autoref{thm:optdata} yields
\begin{equation}\label{eq:max reg estimate blow-up}
    \begin{aligned}
        \| (\vair,\vocn) \|_{\E_{1,\mu_c}(T_0,T)}
        &\leq M \bigl( \| \vair(T_0) \|_{\rB_{qp,\sigmabar}^{\nicefrac{2}{q}}(\Omegaatm)} + \| \fair + G^\air(\vair,\vair) \|_{\E_{0,\mu_c,\air}} + \| B(\vair,\vocn) \|_{\F_{\mu,\air}} \\
        & \quad + \| \vocn(T_0) \|_{\rB_{qp,\sigmabar}^{\nicefrac{2}{q}}(\Omegaocn)} + \| \focn + G^\ocn(\vocn,\vocn) \|_{\E_{0,\mu_c,\ocn}} + \| B(\vair,\vocn) \|_{\F_{\mu,\ocn}}\bigr).
    \end{aligned}
\end{equation}

Recalling the continuity of $G \colon \rX_\beta \times \rX_\beta \to \rX_0$, we derive from \eqref{eq:emeddingmaxreg} the estimate 
\begin{equation}\label{eq:est rhs blow-up}
    \| G(v,v) \|_{\E_{1,\mu_c}(T_0,T)} \le C \| v \|_{\rL_\tau^{2p}(T_0,T;\rX_\beta)}^2 \le C \| v \|_{\rLp(T_0,T;\rX_{\mu_c})} \| v \|_{\E_{1,\mu_c}(T_0,T)},
\end{equation}
where $C > 0$ is independent of $T \in (T_0,t_+)$.
It remains to consider the boundary terms.
For the first one, exploiting \eqref{eq:est Lp Besov} as well as \eqref{eq:emeddingmaxreg}, we obtain
\begin{equation}\label{eq:est Besov norm blow-up}
    \begin{aligned}
        &\quad \| B(\vair,\vocn) \|_{\rLp_{\mu_c}(T_0,T;\rB_{qq}^{1 - \nicefrac{1}{q}}(G))}\\
        & \le  C\bigl(\| \vair \|_{\rL_\tau^{2p}(0,T;\rX_\beta^\air)}^2 + \| \vocn \|_{\rL_\tau^{2p}(0,T;\rX_\beta^\ocn)}^2\bigr)\\
        &\le C\bigl(\| \vair \|_{\rLp(T_0,T;\rX_{\mu_c}^\air)} \| \vair \|_{\E_{1,\mu_c,\air}(T_0,T)}
        + \| \vocn \|_{\rLp(T_0,T;\rX_{\mu_c}^\ocn)} \| \vocn \|_{\E_{1,\mu_c,\ocn}(T_0,T)}\bigr).
    \end{aligned}
\end{equation}

The treatment of the second term is more involved. 
Invoking \eqref{eq:est of second factor of second part of boundary norm} as well as \eqref{eq:emeddingmaxreg}, we find that the second factor of 
\eqref{eq:est Triebel-Lizorkin Lq} satisfies
\begin{equation}\label{eq:est second factor Triebel-Lizorkin}
    \begin{aligned}
         \| \vair \|_{\rL_\tau^{2p}(T_0,T;\rW^{\nicefrac{1}{2q},2q}(\Omegaair))} & + \| \vocn \|_{\rL_\tau^{2p}(T_0,T;\rW^{\nicefrac{1}{2q},2q}(\Omegaocn))}\\
        &\le C(\| \vair \|_{\rLp(T_0,T;\rX_{\mu_c}^\air)}^{\nicefrac{1}{2}} \| \vair \|_{\E_{1,\mu_c,\air}(T_0,T)}^{\nicefrac{1}{2}} + \| \vocn \|_{\rLp(T_0,T;\rX_{\mu_c}^\ocn)}^{\nicefrac{1}{2}} 
\| \vocn \|_{\E_{1,\mu_c,\ocn}(T_0,T)}^{\nicefrac{1}{2}}).
    \end{aligned}
\end{equation}
For the first factor of \eqref{eq:est Triebel-Lizorkin Lq}, we deduce from \eqref{eq:emeddingmaxreg} and \eqref{eq:rel emb Bessel pot Triebel-Lizorkin} joint with the assumption $q \ge 2$ that
\begin{equation*}
    (\rLp(0,T; \rX_{\mu_c}),\E_{1,\mu_c})_{\nicefrac{1}{2}} \hra \rH^{\nicefrac{(1-\alpha)}{2},p}_{\tau}(0,T ; \rH^{\mu_c + \alpha,q}(\Omega;\R^2)) \hra \rF_{pq,\tau}^{\nicefrac{(1-\alpha)}{2}}(0,T;\rH^{\nicefrac{1}{p} + \nicefrac{1}{q} + \alpha,q}(\Omega;\R^2))
\end{equation*}
is valid for all $\alpha \in (0,1)$. 
The latter space is supposed to embed into $\rF_{2p q,\tau}^{\nicefrac{1}{2} - \nicefrac{1}{2q}}(0,T;\rW^{\nicefrac{1}{2q},2q}(\Omega;\R^2))$.
For the time component, this requires $\alpha \le \frac{1}{q} - \frac{1}{p}$.
On the other hand, for the spatial component, we obtain the embedding $\rH^{\nicefrac{1}{p} + \nicefrac{1}{q} + \alpha,q}(\Omega;\R^2) \hra \rW^{\nicefrac{1}{p} + \nicefrac{1}{q} + \alpha,q}(\Omega;\R^2) 
\hra \rW^{\nicefrac{1}{2q},2q}(\Omega;\R^2)$ provided $\alpha \ge \frac{1}{q} - \frac{1}{p}$.
Hence, choosing $\alpha = \frac{1}{q} - \frac{1}{p}$, we see that $\alpha \in (0,1)$. Consequently, 
\begin{equation}\label{eq:est first factor Triebel-Lizorkin}
    \begin{aligned}
 &\quad \| \vair \|_{\rF_{2p q,\tau}^{\nicefrac{1}{2} - \nicefrac{1}{2q}}(T_0,T;\rW^{\nicefrac{1}{2q},2q}(\Omegaair))} + \| \vocn \|_{\rF_{2p q,\tau}^{\nicefrac{1}{2} - \nicefrac{1}{2q}}(T_0,T;\rW^{\nicefrac{1}{2q},2q}(\Omegaocn))}\\
  &\le C\bigl(\| \vair \|_{\rLp(T_0,T;\rX_{\mu_c}^\air)}^{\nicefrac{1}{2}} \| \vair \|_{\E_{1,\mu_c,\air}(T_0,T)}^{\nicefrac{1}{2}} + \| \vocn \|_{\rLp(T_0,T;\rX_{\mu_c}^\ocn)}^{\nicefrac{1}{2}} \| \vocn \|_{\E_{1,\mu_c,\ocn}(T_0,T)}^{\nicefrac{1}{2}}\bigr).
    \end{aligned}
\end{equation}
An insertion of \eqref{eq:est second factor Triebel-Lizorkin} and \eqref{eq:est first factor Triebel-Lizorkin} into \eqref{eq:est Triebel-Lizorkin Lq} results in
\begin{equation}\label{eq:est Triebel-Lizorkin norm blow-up}
    \begin{aligned}
        &\quad \| B(\vair,\vocn) \|_{\rF_{pq,\mu}^{\nicefrac{1}{2}-\nicefrac{1}{2q}}(T_0,T;\rLq(G))}\\
        &\le C\bigl(\| \vair \|_{\rLp(T_0,T;\rX_{\mu_c}^\air)} \| \vair \|_{\E_{1,\mu_c,\air}(T_0,T)} + \| \vocn \|_{\rLp(T_0,T;\rX_{\mu_c}^\ocn)} \| \vocn \|_{\E_{1,\mu_c,\ocn}(T_0,T)}\bigr).
    \end{aligned}
\end{equation}
Plugging the estimates \eqref{eq:est rhs blow-up}, \eqref{eq:est Besov norm blow-up} as well as \eqref{eq:est Triebel-Lizorkin norm blow-up} into \eqref{eq:max reg estimate blow-up}, we obtain
\begin{equation}\label{eq:combined max reg est blow-up}
    \begin{aligned}
        &\quad \| (\vair,\vocn) \|_{\E_{1,\mu_c}(T_0,T)}\\
        &\le M(\| \vair(T_0) \|_{\rB_{qp,\sigmabar}^{\nicefrac{2}{q}}(\Omegaatm)} + \| \vocn(T_0) \|_{\rB_{qp,\sigmabar}^{\nicefrac{2}{q}}(\Omegaocn)} + \| \fair \|_{\E_{0,\mu_c,\air}(T_0,T)} + \| \focn \|_{\E_{0,\mu_c,\ocn}(T_0,T)}\\
        &\quad + C(\| \vair \|_{\rLp(T_0,T;\rX_{\mu_c}^\air)} \| \vair \|_{\E_{1,\mu_c,\air}(T_0,T)} + \| \vocn \|_{\rLp(T_0,T;\rX_{\mu_c}^\ocn)} \| \vocn \|_{\E_{1,\mu_c,\ocn}(T_0,T)})).
    \end{aligned}
\end{equation}
Let now $\eta \coloneqq \frac{1}{2MC}$ and $t_0 \in (T_0,t_+)$ close to $t_+$ such that $\| v \|_{\rLp(T_0,T;\rX_{\mu_c})} \le \eta$.
By \eqref{eq:combined max reg est blow-up}, for any $T \in (t_0,t_+)$, the norm $\| (\vair,\vocn) \|_{\E_{1,\mu_c}(T_0,T)}$ can be bounded by
\begin{equation*}
    2 M\bigl(\| \vair(T_0) \|_{\rB_{qp,\sigmabar}^{\nicefrac{2}{q}}(\Omegaatm)} + \| \vocn(T_0) \|_{\rB_{qp,\sigmabar}^{\nicefrac{2}{q}}(\Omegaocn)} + \| \fair \|_{\E_{0,\mu_c,\air}(T_0,T)} + \| \focn \|_{\E_{0,\mu_c,\ocn}(T_0,T)}\bigr).
\end{equation*}
Since $(\vair,\vocn)|_{(t_0,t_+)} \in \E_{1,\mu_c}(t_0,t_+) \hra \mathrm{BUC}([t_0,t_+];\rB_{qp,\sigmabar}^{\nicefrac{2}{q}}(\Omegaatm) \times \rB_{qp,\sigmabar}^{\nicefrac{2}{q}}(\Omegaocn))$, the solution can be continued beyond $t_+$, which yields  a contradiction.
\end{proof}

\section{Global well-posedness}\label{sec:global}

Throughout this section, assume $0< T < \infty$, $p=q=2$, which implies $\mu_c = 1$, and let $(\vair, \vocn)$ be  a solution to \eqref{eq:CAO}, or, equivalently, to \eqref{eq:CAOsimplified}, satisfying
\begin{equation}\label{eq:Assumptionsapriori}
	\begin{aligned}
		(\vair, \vocn) \in \E_{1,1,T},
	\end{aligned}
\end{equation}
with initial data and forcing terms
\begin{equation}\label{eq:Assumptionsapriori2}
    \begin{aligned}
        (\vair_0, \vocn_0)  \in \rH_{\sigmabar}^1(\Omega^\atm) \times \rH_{\sigmabar}^1(\Omega^\ocn) \quad \text{and} \quad 
        ( \fair,  \focn) &\in  \E_{0,1,T},
    \end{aligned}
\end{equation}
where $_{\sigmabar}$ again indicates the intersection with $\rL_{\sigmabar}^2$.
First, we recall from \cite{HK:16} the following cancellation law.

\begin{lem}[\hspace{-0.07em}{\cite[Lemma~6.3]{HK:16}}]\label{lem:canc law}
Let $p,q,r \in (1,\infty)$ and $g \colon \Omega \to G$. Then 
	\begin{equation*}
		\int_\Omega (\vtilde \cdot \nablaH g + w g_z) \cdot |g|^{q-2} g \d(\xH,z) = 0 \quad \text{and} \quad \int_\Omega (\vbar \cdot \nablaH g) \cdot |g|^{q-2} g \d(\xH,z) = 0
	\end{equation*}
provided the integrals exist.
\end{lem}

Next, we consider energy estimates in the situation of fully nonlinear interface  conditions. 

\begin{lem}[Energy estimates]\label{lem:energyestimates}
Let $ (\vair, \vocn), (\vair_0, \vocn_0), (\fair, \focn)$ be as in \eqref{eq:Assumptionsapriori} and \eqref{eq:Assumptionsapriori2}. Then there exists a continuous function $B_1$ on $[0,T]$,
depending on $\| \vair_0 \|_{\rL^2(\Omegaair)}$, $\| \vocn_0 \|_{\rL^2(\Omegaocn)}$, $\| f^{\air} \|_{\rL^2(0,T;\rL^2(\Omegaair))}$, $\| f^{\ocn} \|_{\rL^2(0,T;\rL^2(\Omegaocn))}$ and $t \in [0,T]$, such that
\begin{equation*}
\|\vair \|^2_{\rL^2(\Omegaair)} + \|\vocn \|^2_{\rL^2(\Omegaocn)} +  \int_0^t \| \nabla \vair \|^2_{\rL^2(\Omegaair)} \d s + \int_0^t \| \nabla \vocn \|^2_{\rL^2(\Omegaocn)} \d s + 
\int_0^t \| V\|_{\rL^3(G)} \d s  \leq B_1(t). 
\end{equation*}
\end{lem}

\begin{proof}
We multiply equations \eqref{eq:CAOsimplified}$_1$ and \eqref{eq:CAOsimplified}$_3$ by $\vair$ and $\vocn$ and integrate over $\Omegaair$ and $\Omegaocn$, respectively. Adding the resulting equations, we obtain 
\begin{equation*}
	\begin{aligned}
        &\quad \frac{1}{2}\dt \bigl(\| \vair \|^2_{\rL^2(\Omegaair)} + p_s^2\| \vocn \|^2_{\rL^2(\Omegaocn)}  \bigr)
        +C_a \| \nabla \vair \|^2_{\rL^2(\Omegaair)} +   p_s^2  \| \nabla \vocn \|^2_{\rL^2(\Omegaocn)} + p_s^2\| V \|^3_{\rL^3(G)} \\
        &\leq C \bigl(\| \vair \|^2_{\rL^2(\Omegaair)} +  \| \vocn \|^2_{\rL^2(\Omegaocn)} + \| f^{\air} \|^2_{\rL^2(\Omegaair)} + \| f^{\ocn}\|^2_{\rL^2(\Omegaocn)} \bigr),
	\end{aligned}
\end{equation*} 
where $C_a := \min \{ 1, p_a^2\} >0$. 
An integration in time and Gronwall's inequality yield the claim.
\end{proof}
We proceed by taking the vertical average in \eqref{eq:CAOsimplified}, resulting in
\begin{equation*}
\left \{
	\begin{aligned} 
        \partial_{t}\vairbar - \DeltaH\vairbar + \nablaH \Phi_s 
        &= \overline{f}^{\air}- \vairbar\cdot \nablaH \vairbar -  \int_{0}^{1} \left ( \vairtilde \cdot \nablaH \vairtilde + \divH(\vairtilde) \vairtilde \right ) \d p   + p_s^2\partial_{p}\vair|_{\Gammaairu}\\
        \divH \vairbar &= 0,\\
	\end{aligned} 
	\right.
\end{equation*}
in $G$ as well as the equality
\begin{equation}\label{eq:vairtilde}
    \begin{aligned}
        &\quad \partial_{t}\vairtilde - \Delta^{\atm} \vairtilde 
        + \vairtilde \cdot \nablaH \vairtilde + \omegaair \partial_{p}\vairtilde
        + \vairbar \cdot \nablaH \vairtilde\\
        &= \ftilde^{\air}- \vairtilde\cdot \nablaH \vairbar +  \int_{0}^{1} \left ( \vairtilde \cdot \nablaH \vairtilde + \divH (\vairtilde) \vairtilde \right ) \d p -\partial_{p}\vair|_{\Gammaairu}.
    \end{aligned}
\end{equation}
in $G \times (p_s,p_a)$.
Analogously, we obtain for the oceanic part
\begin{equation*}
\left \{
	\begin{aligned} 
        \partial_{t}\vocnbar - \DeltaH\vocnbar + \nablaH \pi_s 
        &= \overline{f}^{\ocn}- \vocnbar\cdot \nablaH \vocnbar - \int_{-1}^{0} \left ( \vocntilde \cdot \nablaH \vocntilde + \divH (\vocntilde) \vocntilde \right ) \d z +
        \partial_{z}\vocn|_{\Gau}  \\
        \divH \vocnbar &= 0,\\
	\end{aligned} 
	\right.
\end{equation*}
in $G$, and in $G \times (-1,0)$, we get
\begin{equation}\label{eq:vocntilde}
    \begin{aligned}
        &\quad \partial_{t}\vocntilde - \Delta \vocntilde + \vocntilde \cdot \nablaH \vocntilde + \wocn \partial_{z}\vocntilde + \vocnbar \cdot \nablaH \vocntilde \\
        &= \ftilde^{\ocn}- \vocntilde\cdot \nablaH \vocnbar  + \int_{-1}^{0} \left ( \vocntilde \cdot \nablaH \vocntilde + \divH (\vocntilde) \vocntilde \right ) \d z
        -\partial_{z}\vocn|_{\Gau}.
    \end{aligned}
\end{equation}

\begin{prop}[$\rL_t^\infty \rH^{1}_{\bold{x}}$-$\rL^2_t \rH^{2}_{\bold{x}}$-Estimates] \label{prop:linftyl2} \mbox{}\\
Let $ (\vair, \vocn)$, $(\vair_0, \vocn_0)$ and $(\fair, \focn)$ be as in \eqref{eq:Assumptionsapriori} and \eqref{eq:Assumptionsapriori2}. Then there exists a 
continuous function $B_2$ on~$[0,T]$, depending on $\| \vair_0 \|_{\rH^1(\Omegaair)}$, $\| \vocn_0 \|_{\rH^1(\Omegaocn)}$, $\| f^{\air} \|_{\rL^2(0,T;\rL^2(\Omegaair))}$, $\|\focn \|_{\rL^2(0,T;\rL^2(\Omegaocn))}$ and $t$, 
such that
\begin{equation*}
    \| \nabla \vair \|^2_{\rL^2(\Omegaair)} + \| \nabla \vocn \|^2_{\rL^2(\Omegaocn)} + \int_0^t \| \Delta \vair \|^2_{\rL^2(\Omegaair)} \d s  + \int_0^t \| \Delta \vocn \|^2_{\rL^2(\Omegaocn)} \d s  \leq B_2(t), \quad \text{for $t \in [0,T]$}.
\end{equation*}
\end{prop}

\begin{proof} We subdivide the proof in several steps.
\setcounter{stp}{0}
\begin{step}[Estimates for $(\vairbar, \vocnbar)$ in $\rL^{\infty}_t \rH^{1}_{xy} \cap \rL^2_t \rH^2_{xy}$] \label{aprioristep1} \mbox{}\\
As the equations for $\vairbar$ and $\vocnbar$ do not comprise boundary conditions, the estimate claimed can deduced essentially as in 
\cite{HK:16} or \cite{GGHHK:20}. 
Note that, however, we do not  fully absorb the $\rH^2$-norm.
The resulting estimate reads as
\begin{equation*}
	\begin{aligned}
8\partial_t \| \nablaH \overline{v}\|^2_{\rL^2(G)}  + \| \DeltaH \vbar \|^2_{\rL^2(\Omega)} +\|\nablaH \pi \|^2_{\rL^2(G)}
        &\leq C\| |\tilde{v}| |\nabla \tilde{v}|\|^2_{\rL^2(\Omega)} + C \bigl(\| v \|^2_{\rH^{1}(\Omega)}  +\| v \|^2_{\rL^2(\Omega)} \bigr)\| v \|^2_{\rL^2(\Omega)} \\
        & \quad+ C \bigl( \| v \|_{\rL^2(\Omega)} + \| v \|^2_{\rL^2(\Omega)} \bigr) \bigl( \| v \|_{\rH^{1}(\Omega)} + \| v \|^2_{\rH^{1}(\Omega)}\bigr) \| \nablaH \vbar\|^2_{\rL^2(G)} \\
         &\quad + C \| v_z \|^2_{\rL^2(\Omega)} + C \eps \| \nabla v_z \|^2_{\rL^2(\Omega)}+C \| f \|^2_{\rL^2(\Omega)}.
    \end{aligned}
\end{equation*}
\end{step}

\begin{step}[Estimates for $(\vairtilde, \vocntilde)$ in $\rL_t^\infty  \rL^4_{xyp} \times \rL_t^\infty \rL^4_{xyz}$]\label{aprioristep3} \mbox{}\\
We multiply \eqref{eq:vairtilde} and \eqref{eq:vocntilde} by $|\vairtilde|^2\vairtilde$ as well as $|\vocntilde|^2\vocntilde$ and integrate over $\Omegaair$ and  $\Omegaocn$ to obtain 
\begin{equation*}
    \begin{aligned}
        &\quad \frac{1}{4} \dt \| \vairtilde \|^4_{\rL^4(\Omegaair)}  
         + C_a \bigl(\frac{1}{2} \| \nabla |  \vairtilde |^2\|^2_{\rL^2(\Omegaair)} + \| | \vairtilde | |\nabla  \vairtilde| \|^2_{\rL^2(\Omegaair)} \bigr)+ p_s^{2} \int_G V|V| |\vairtilde|^2\vairtilde \d \xH  \\
        &\leq \int_{\Omegaair} \tilde{f}^{\air} \cdot |\vairtilde|^2\vairtilde \d(\xH,p) +  \int_{\Omegaair} \int_{0}^{1} (\vairtilde \cdot \nablaH \vairbar + \divH(\vairtilde) \vairtilde ) \d p \cdot 
        |\vairtilde|^2\vairtilde \d(\xH,p)\\
        &\quad -\int_{\Omegaair} (\vairtilde \cdot \nablaH \vairbar) \cdot |\vairtilde|^2\vairtilde \d(\xH,p) -  \int_{\Omegaair} \vairtilde_p |_{\Gammaairu} \cdot |\vairtilde|^2\vairtilde \d(\xH,p),
    \end{aligned}
\end{equation*}
with $C_a := \min \{ 1, p_a^2\}$, as well as
\begin{equation*}
    \begin{aligned}
        &\quad \frac{1}{4} \dt \| \vocntilde \|^4_{\rL^4(\Omegaocn)} + \frac{1}{2} \| \nabla |  \vocntilde |^2\|^2_{\rL^2(\Omegaocn)} + \| | \vocntilde | |\nabla  \vocntilde| \|^2_{\rL^2(\Omegaocn)} - 
        \int_G V|V| |\vocntilde|^2\vocntilde \d \xH \\
        &= \int_{\Omegaocn} \tilde{f}^{\ocn} \cdot |\vocntilde|^2\vocntilde \d(\xH,z) + \int_{\Omegaocn} \int_{-1}^{0} (\vocntilde \cdot \nablaH \vocnbar + \divH(\vocntilde) \vocntilde ) \d z \cdot |\vocntilde|^2\vocntilde \d(\xH,z) \\
        & \quad - \int_{\Omegaocn} (\vocntilde \cdot \nablaH \vocnbar) \cdot |\vocntilde|^2\vocntilde \d(\xH,z) -\int_{\Omegaocn} \vocntilde_z |_{{\Gau}} \cdot |\vocntilde|^2\vocntilde \d(\xH,z).
    \end{aligned}
\end{equation*}
A multiplication of the oceanic equation by $p_s^2$ and an addition of it to the atmospheric inequality leads to
\begin{equation}
    \begin{aligned}\label{eq:tildediffrence}
&\quad \frac{1}{4}\dt \bigl( \| \vairtilde \|^4_{\rL^4(\Omegaair)}+p_s^{2}\| \vocntilde \|^4_{\rL^4(\Omegaocn)} \bigr) + \frac{1}{2} \bigl(C_a \| \nabla |  \vairtilde |^2\|^2_{\rL^2(\Omegaair)}   
+ p_s^{2}\| \nabla |  \vocntilde |^2\|^2_{\rL^2(\Omegaocn)}\bigr) + C_a\| | \vairtilde | |\nabla  \vairtilde| \|^2_{\rL^2(\Omegaair)} \\ 
&\quad + p_s^{2}\| | \vocntilde | |\nabla  \vocntilde| \|^2_{\rL^2(\Omegaocn)} + 
\frac{p_s^2}{2} \bigl (  \| |\vairtilde||\tilde{V} |V|^{\nicefrac{1}{2}} \|^2_{\rL^2(G)} + \| |\vocntilde||\tilde{V} |V|^{\nicefrac{1}{2}} \|^2_{\rL^2(G)}  \bigr )\\
&\le  \int_{\Omegaair} \tilde{f}^{\air} \cdot |\vairtilde|^2\vairtilde \d(\xH,p) + 
p_s^{2}\int_{\Omegaocn} \tilde{f}^{\ocn} \cdot |\vocntilde|^2\vocntilde \d(\xH,z)\\
&\quad + \int_{\Omegaair} \int_{0}^{1} (\vairtilde \cdot \nablaH \vairbar + \divH(\vairtilde) \vairtilde ) \d p \cdot |\vairtilde|^2\vairtilde \d(\xH,p) \\
&\quad  + p_s^2 \int_{\Omegaocn} \int_{-1}^{0} (\vocntilde \cdot \nablaH \vocnbar + \divH(\vocntilde) \vocntilde ) \d z \cdot |\vocntilde|^2\vocntilde \d(\xH,z) 
-\int_{\Omegaair} (\vairtilde \cdot \nablaH \vairbar) \cdot |\vairtilde|^2\vairtilde \d(\xH,p) \\
&\quad -p_s^{2}\int_{\Omegaocn} (\vocntilde \cdot \nablaH \vocnbar) \cdot |\vocntilde|^2\vocntilde \d(\xH,z)
        - \int_{\Omegaair} \vairtilde_p |_{\Gammaairu} \cdot |\vairtilde|^2\vairtilde \d(\xH,p)  - p_s \int_{\Omegaocn} \vocntilde_z |_{{\Gau}} \cdot |\vocntilde|^2\vocntilde \d(\xH,z)  \\
        &\quad -p_s^2 \int_G \bar{V}|V|\left (|\vairtilde|^2\vairtilde  -   |\vocntilde|^2\vocntilde \right ) \d \xH.
    \end{aligned}
\end{equation}
Here and in the following we employ the splitting $V = \bar{V} + \Tilde{V}$ where $\bar{V}:= \vbaratm - \vbarocn$ and $\tilde{V}:= \vairtilde |_{\Gammaairu} - \vocntilde |_{\Gau}$.
In view of the results in \cite{HK:16}, we only need to estimate the last term above. We obtain
\begin{equation*}
    \begin{aligned}
        \int_G \bar{V}|V|\left (|\vairtilde|^2\vairtilde    - |\vocntilde|^2\vocntilde \right ) \d \xH = -\int_G \tilde{V}|\bar{V}|V \left ( |\vairtilde|^2 + \vairtilde \vocntilde + |\vocntilde|^2 \right ) \d \xH.
    \end{aligned}
\end{equation*}
Hence, it suffices to estimate $\| |\tilde{V}| |\bar{V}| |V| |\vtilde|^2 \|_{\rL^1(G)}$ for $\vocntilde$ and $\vairtilde$. Observe that for $\eps$, $\eps' > 0$, we get
\begin{equation*}
    \begin{aligned}
        \int_G |\tilde{V}| |\bar{V}| |V| |\vtilde|^2 \d \xH 
        &\leq \eps \| |\vtilde| |\Tilde{V}| |V|^{\nicefrac{1}{2}} \|^2_{\rL^2(G)} + C\| |\vtilde| |\bar{V}||V|^{\nicefrac{1}{2}} \|^2_{\rL^2(G)} \quad \text{and}\\
        \| |\vtilde| |\bar{V}||V|^{\nicefrac{1}{2}} \|^2_{\rL^2(G)} &\leq \eps' \| \vtilde \|^4_{\rL^8(G)} + C \| |\bar{V}||V|^{\nicefrac{1}{2}}\|^4_{\rL^{\nicefrac{8}{3}}(G)} \\ 
        &\le \eps' \bigl( \| \nabla |\vtilde|^2 \|^2_{\rL^2(\Omega)} + \|\vtilde\|^4_{\rL^4(\Omega)} \bigr) +C \| \bar{V} \|^6_{\rL^4(G)} + C \| |\bar{V}| |\tilde{
        V}|^{\nicefrac{1}{2}}\|^4_{\rL^{\nicefrac{8}{3}}(G)} 
    \end{aligned}
\end{equation*}
by using 
\begin{equation}
    \begin{aligned}\label{eq:vtilde3}
\| \vtilde \|^4_{\rL^8(G)} = \| |\vtilde |^2 \|^2_{\rL^4(G)} \leq C\| |\vtilde|^2\|^2_{\rW^{\nicefrac{1}{4},4}(\Omega)} \leq C \| |\vtilde|^2 \|^2_{\rH^1(\Omega)} 
\leq C \bigl( \| \nabla |\vtilde|^2 \|^2_{\rL^2(\Omega)} + \| |\vtilde|^2 \|^2_{\rL^2(\Omega)}\bigr).
    \end{aligned}
\end{equation}
We further get $\|\bar{V}\|^6_{\rL^4(G)} \leq \eps'' \| \bar{V}\|^2_{\rH^2(G)} + C\| \bar{V}\|^{18}_{\rL^2(G)} \leq \eps'' \| \bar{V}\|^2_{\rH^2(G)} + \| \bar{V}\|^{26}_{\rL^2(G)} +C$ for suitable $\eps'' > 0$ by Ladyzhenskaya's and Young's inequality, so it remains to estimate $\| |\bar{V}| |\tilde{V}|^{\nicefrac{1}{2}}\|^4_{\rL^{\nicefrac{8}{3}}(G)}$. Employing H\"older's, Young's and Ladyzhenskaya's inequality, we obtain for some $\eps'''$, $\eps'''' > 0$ the estimate
\begin{equation*}
\| |\bar{V}||\tilde{V}|^{\nicefrac{1}{2}}\|^4_{\rL^{\nicefrac{8}{3}}(G)} \leq \eps'''\| \tilde{V} \|^4_{\rL^8(G)} + C\| \bar{V}\|^8_{\rL^{\nicefrac{16}{5}}(G)} \leq \eps'''\| \tilde{V} \|^4_{\rL^8(G)} 
+ C\eps'''' \| \bar{V} \|^2_{\rH^2(G)} + C \|\bar{V}\|^{26}_{\rL^2(G)}.
\end{equation*}
Summarizing, for $\eps'$, $\eps'' > 0$, we proved
\begin{equation}\label{eq:vtilde4}
 \| |\vtilde| |\bar{V}||V|^{\nicefrac{1}{2}} \|^2_{\rL^2(G)} \leq \eps' \bigl( \| \nabla |\vtilde|^2 \|^2_{\rL^2(\Omega)} + \| |\vtilde|^2 \|^2_{\rL^2(\Omega)}\bigr) + 
\eps''\| \bar{V} \|^2_{\rH^2(G)} + C \| \bar{V} \|^{26}_{\rL^2(G)} +C.
\end{equation}
\end{step}

\begin{step}[Estimates for $(\vair_p, \vocn_z)$  in $\rL_t^\infty  \rL^2_{xyp} \cap  \rL_t^2 \rH^{1}_{xyp} \times \rL_t^\infty \rL^2_{xyz} \cap  \rL_t^2 \rH^{1}_{xyz}$]\label{aprioristep2} \mbox{} \\
Multiplying \eqref{eq:CAOsimplified}$_1$, \eqref{eq:CAOsimplified}$_3$ by $-\vair_{pp}$, $-\vocn_{zz}$, and integrating over $\Omegaair$ and $\Omegaocn$, respectively, we obtain by~\autoref{lem:canc law}
\begin{equation*}
\begin{aligned}
 &\enspace \frac{1}{2} \dt \| \vair_p \|^2_{\rL^2(\Omegaair)} + \int_G \dt \vair V |V| \d \xH + C_a\|\nabla \vair_p \|^2_{\rL^2(\Omegaair)}+2\int_{\Omegaair}p \partial_p \vair \partial_{pp}\vair \d(\xH,p)+ 2\int_G \nablaH \vair \nablaH V |V| \d \xH \\
        &\le \int_G \nablaH \Phi_s \cdot \vair_p|_{\Gammaairb} \d \xH - \int_{\Omegaair} \fair \cdot \vair_{pp} \d(\xH,p) - \int_{\Omegaair} (\vair_p \cdot \nablaH \vair) \vair_p \d(\xH,p)\\
        &\quad + \int_{\Omegaair} \divH(\vair) \vair_p \cdot \vair_p \d(\xH,p)
        + \int_G (\vair \cdot \nablaH \vair) V|V| \d \xH
    \end{aligned}
\end{equation*}
as well as 
\begin{equation*}
    \begin{aligned}
        &\quad \frac{1}{2} \dt \| \vocn_z \|^2_{\rL^2(\Omegaocn)} - \int_G \dt \vocn V |V| \d \xH + \|\nabla \vocn_z \|^2_{\rL^2(\Omegaocn)} -2 \int_G \nablaH \vocn \nablaH V |V| \d \xH \\
        &= \int_G \nablaH \pi_s \cdot \vocn_z|_{\Gau} \d \xH - \int_{\Omegaocn} \focn \cdot \vocn_{zz} \d (\xH,z)
         - \int_{\Omegaocn} (\vocn_z \cdot \nablaH \vocn) \vocn_z \d(\xH,z)\\
         &\quad + \int_{\Omegaocn} \divH(\vocn) \vocn_z \cdot \vocn_z \d(\xH,z)
        - \int_G (\vocn \cdot \nablaH \vocn) V|V| \d \xH.
    \end{aligned}
\end{equation*}
Adding the atmospheric and oceanic part, and using the relation $\dt |V|^3 = 3 V |V| \dt V$, we find that
\begin{equation*}
\begin{aligned} 
 &\quad \dt \bigl( \frac{1}{2}\| \vair_p \|^2_{\rL^2(\Omegaair)}+ \frac{1}{2}\| \vocn_z \|^2_{\rL^2(\Omegaocn)} \bigr)+ \frac{1}{3} \dt  \|V \|^3_{\rL^3(G)} + 
C_a \| \nabla \vair_p \|^2_{\rL^2(\Omegaair)} + \|\nabla \vocn_z \|^2_{\rL^2(\Omegaocn)} + 2\| |\nablaH V| |V|^{\nicefrac{1}{2}} \|^2_{\rL^2(G)} \\ 
        &\leq \int_G  \nablaH \Phi_s \cdot \vair_p|_{\Gammaairb} \d \xH +  \int_G \nablaH \pi_s \cdot \vocn_z|_{\Gau} \d \xH  - \int_{\Omegaair} \fair \cdot \vair_{pp} \d(\xH,p) - \int_{\Omegaocn} \focn \cdot \vocn_{zz} \d(\xH,z)\\
        &\quad - \int_{\Omegaair} (\vair_p \cdot \nablaH \vair) \vair_p \d(\xH,p) + \int_{\Omegaair} \divH(\vair) \vair_p \cdot \vair_p \d(\xH,p) - \int_{\Omegaocn} (\vocn_z \cdot \nablaH \vocn) \vocn_z \d(\xH,z)\\
        &\quad +  \int_{\Omegaocn} \divH(\vocn) \vocn_z \cdot \vocn_z \d(\xH,z) +\int_G\left ( (\vair \cdot \nablaH \vair)  -  (\vocn \cdot \nablaH \vocn) \right ) |V|V \d \xH -2\int_{\Omegaair}p \partial_p \vair \partial_{pp}\vair \d(\xH,p).
    \end{aligned}
\end{equation*}
We only focus on the estimate of the last two terms on the right-hand side above. 
Using the relations $\vair \cdot \nablaH \vair - \vocn \cdot \nablaH \vocn = V \cdot \nablaH \vair + \vocn \cdot \nablaH V$ and 
$\int_{\Omegaair}p \partial_p \vair \partial_{pp}\vair \d(\xH,p) \le C \| \vair_p \|^2_{\rL^2(\Omegaair)} + \eps \| \nabla \vair_p \|^2_{\rL^2(\Omegaair)}$, we obtain
\begin{equation*}
    \int_G\left (V \nablaH \vair + \vocn \nablaH V \right )|V| V \d \xH = -\int_G \left (2\vair\cdot  \nablaH V + \divH(V) \cdot \vair\right ) |V|V \d \xH + \int_G \vocn \cdot \nablaH V |V|V \d \xH.
\end{equation*}
Hence, there exists $C_1>0$ such that 
\begin{equation*}
 \begin{aligned}
&\quad \bigl | \int_G\left ( (\vair \cdot \nablaH \vair)  -  (\vocn \cdot \nablaH \vocn) \right )|V| V \d \xH \bigr | 
        \leq C \||\vair| | \nablaH V |  |V|V\|_{\rL^1(G)} +  \||\vocn| | \nablaH V |  |V|V\|_{\rL^1(G)} \\ 
&\leq 2 \| | \nablaH V| |V|^{\nicefrac{1}{2}} \|^2_{\rL^2(G)} + C_1 \bigl(\| |\vair| |V|^{\nicefrac{3}{2}} \|^2_{\rL^2(G)} +\| |\vocn| |V|^{\nicefrac{3}{2}} \|^2_{\rL^2(G)} \bigr).
    \end{aligned}
\end{equation*}
Noting that $V = \tilde{V}+\overline{V}$ as well as  $v = \vtilde + \vbar$, we see that 
\begin{equation*}
    \begin{aligned}      
        \| |v| |V|^{\nicefrac{3}{2}} \|^2_{\rL^2(G)} &\leq C \bigl( \| |\vtilde||\tilde{V} |V|^{\nicefrac{1}{2}}\|^2_{\rL^2(G)} +  \| |\vtilde||\Bar{V}|  |V|^{\nicefrac{1}{2}}\|^2_{\rL^2(G)} +   \| |\vbar||\tilde{V}|V|^{\nicefrac{1}{2}}\|^2_{\rL^2(G)}+\| |\vbar||\bar{V}|V|^{\nicefrac{1}{2}}\|^2_{\rL^2(G)} \bigr) \\
        &=: I_1 + I_2 + I_3 + I_4,
    \end{aligned}
\end{equation*}
where $v$ denotes $\vair$ and $\vocn$, respectively. Now, $I_1$ can be absorbed into \eqref{eq:tildediffrence}, whereas $I_2$ was estimated in \eqref{eq:vtilde4}. 
For $I_3$, we notice that $I_3 \leq C (\| |\vbar||\vairtilde|V|^{\nicefrac{1}{2}}\|^2_{\rL^2(G)} +\| |\vbar||\vocntilde|V|^{\nicefrac{1}{2}}\|^2_{\rL^2(G)} )$ and refer to \autoref{aprioristep3}. 
Concerning $I_4$, we observe that 
\begin{equation*}
    \begin{aligned}
            I_4 &\leq C \bigl( \| |\vairbar|^2 |V|^{\nicefrac{1}{2}} \|^2_{\rL^2(G)} + \| |\vairbar||\vocnbar||V|^{\nicefrac{1}{2}} \|^2_{\rL^2(G)}+ \| |\vocnbar|^2 |V|^{\nicefrac{1}{2}}  \|^2_{\rL^2(G)}\bigr) \\
            &\leq C\bigl( \| |\vairbar|^2 |V|^{\nicefrac{1}{2}} \|^2_{\rL^2(G)} + \| |\vocnbar|^2 |V|^{\nicefrac{1}{2}} \|^2_{\rL^2(G)}\bigr), 
    \end{aligned}
\end{equation*}
and, upon further splitting $V$,
\begin{equation*}
    \begin{aligned}
        \| |\vairbar|^2 |V|^{\nicefrac{1}{2}} \|^2_{\rL^2(G)} 
        &\leq C \bigl( \| |\vairbar|^{\nicefrac{5}{2}} \|^2_{\rL^2(G)} + \|| \vocnbar|^{\nicefrac{5}{2}} \|^2_{\rL^2(G)}+\| |\vairbar|^2 |\tilde{V}|^{\nicefrac{1}{2}}  \|^2_{\rL^2(G)} \bigr).
    \end{aligned}
\end{equation*}
Hence, it remains to estimate $J_1 := \| |\vbar|^{\nicefrac{5}{2}}\|_{\rL^2(G)}^2$ as well as $J_2:=  \| |\vbar|^2 |\vtilde|^{\nicefrac{1}{2}}\|_{\rL^2(G)}^2$ for $\vair$ and $\vocn$, respectively.
Ladyzhenskaya's inequality yields  
\begin{equation*}
    J_1 = \| \vbar \|^5_{\rL^5(G)} \leq \| \vbar \|^{\nicefrac{3}{2}}_{\rH^2(G)} \| \vbar \|^{\nicefrac{7}{2}}_{\rL^2(G)} \leq \eps \| \vbar \|^2_{\rH^2(G)} + C\| \vbar\|^{14}_{\rL^2(G)} \leq \eps \| \vbar \|^2_{\rH^2(G)} + \| \vbar\|^{26}_{\rL^2(G)} +C,
\end{equation*}
and analogously
\begin{equation*}
    \begin{aligned}
        J_2 &\leq \| \vtilde \|_{\rL^8(G)} \| \vbar \|^4_{\rL^{\nicefrac{32}{7}}(G)} \leq \eps \| \vtilde \|^4_{\rL^8(G)} + C \| \vbar \|^{\nicefrac{16}{3}}_{\rL^{\nicefrac{32}{7}}(G)} \leq \eps \| \vtilde \|^4_{\rL^8(G)} + \eps' \| \vbar \|^2_{\rH^2(G)} + \| \vbar \|^{26}_{\rL^2(G)} +C,
    \end{aligned}
\end{equation*}
where we refer to \eqref{eq:vtilde3} for the term $\|\vtilde\|^4_{\rL^8(G)}$.

\end{step}
\begin{step}[Adding the above estimates]\label{aprioristep4}\mbox{}\\
Multiplying \eqref{eq:tildediffrence} by $\frac{2C_1}{p_s}$, and adding the estimates from the steps above, we deduce that
\begin{equation*}
\begin{aligned}
&\quad \dt \bigl(8 \| \nablaH \vairbar\|^2_{\rL^2(G)} +8\| \nablaH \vocnbar \|^2_{\rL^2(G)}+\frac{C}{4} \| \vairtilde \|^4_{\rL^4(\Omegaair)}+\frac{Cp_s^2}{4} \| \vocntilde \|^4_{\rL^4(\Omegaocn)}\\
&\quad +\frac{1}{2}\| \vair_p\|^2_{\rL^2(\Omegaair)}+\frac{1}{2}  \| \vocn_z \|^2_{\rL^2(\Omegaocn)} +\frac{1}{3} \| V \|^3_{\rL^3(G)}  \bigr) + \frac{1}{2}\| \nablaH \Phi \|^2_{\rL^2(G)} + \frac{1}{2} \| 
\nablaH \pi_s \|^2_{\rL^2(G)}\\
&\quad + \delta \bigl( \| \DeltaH \vairbar \|^2_{\rL^2(G)} + \| \DeltaH \vocnbar \|^2_{\rL^2(G)}  +  \| \nabla  \vair_p \|^2_{\rL^2(\Omegaair)}+ \| \nabla  \vocn_z \|^2_{\rL^2(\Omegaocn)} \bigr)\\
&\leq K_1(t) \bigl(\| \nablaH \vairbar\|^2_{\rL^2(G)} +\| \nablaH \vocnbar \|^2_{\rL^2(G)}+ \| \vairtilde \|^4_{\rL^4(\Omegaair)}+ \| \vocntilde \|^4_{\rL^4(\Omegaocn)}\\
&\quad +\| \vair_p\|^2_{\rL^2(\Omegaair)}+  \| \vocn_z \|^2_{\rL^2(\Omegaocn)} + \| V \|^3_{\rL^3(G)}  \bigr) +K_2(t)
    \end{aligned}
\end{equation*}
for some  $\delta >0$, and where $K_1$ and $K_2$ are functions integrable with respect to time by \autoref{lem:energyestimates}. 
Finally, Gronwall's inequality yields the existence of $b_1 = b_1(\|\vair_0\|_{\rH^1(\Omegaair)}, \| \vocn_0 \|_{\rH^1(\Omegaocn)}, \| \fair \|_{\rL^2(\Omegaair)}, \| \focn \|_{\rL^2(\Omegaocn)}, t)$ with
\begin{equation*}
\begin{aligned}
&\| \nablaH \vairbar\|^2_{\rL^2(G)} + \| \nablaH \vocnbar \|^2_{\rL^2(G)}+ \| \vairtilde \|^4_{\rL^4(\Omegaair)}+ \| \vocntilde \|^4_{\rL^4(\Omegaocn)}  
+\| \vair_p\|^2_{\rL^2(\Omegaair)}+  \| \vocn_z \|^2_{\rL^2(\Omegaocn)} + \| V \|^3_{\rL^3(G)} \\
& +\| \nablaH \Phi \|^2_{\rL^2(0,t;\rL^2(G))} + \| \nablaH \pi_s \|^2_{\rL^2(0,t;\rL^2(G))} + \| \DeltaH \vairbar \|^2_{\rL^2(0,t;\rL^2(G))} + \| \DeltaH \vocnbar \|^2_{\rL^2(0,t;\rL^2(G))} \\
& +  \| \nabla  \vair_p \|^2_{\rL^2(0,t;\rL^2(\Omegaair))}+ \| \nabla  \vocn_z \|^2_{\rL^2(0,t;\rL^2(\Omegaocn))} \leq b_1.
\end{aligned}
\end{equation*}
\end{step}

\begin{step}[Estimates for $(\vair,\vocn)$ in $\rL_t^\infty \rH^{1}_{xyp} \cap \rL^2_t \rH^{2}_{xyp} \times \rL_t^\infty \rH^{1}_{xyz} \cap \rL^2_t \rH^{2}_{xyz}$]\label{aprioristep5} \mbox{} \\
A multiplication of \eqref{eq:CAOsimplified}$_1$ and  \eqref{eq:CAOsimplified}$_3$ by $-\Delta \vair$ and  $-\Delta \vocn$ and an integration over $\Omegaair$, $\Omegaocn$ yields
\begin{equation*}
\begin{aligned}
&\quad \partial_t \| \nabla \vair \|^2_{\rL^2(\Omegaair)} + \| \Delta  \vair\|^2_{\rL^2(\Omegaair)} + \int_G \dt \vair |V|V \d \xH \\
&\leq C \bigl(\| \vairbar \cdot  \nablaH \vairbar \|^2_{\rL^2(G)} + \| \vairbar \cdot  \nablaH \vairtilde \|^2_{\rL^2(\Omegaair)} +  \| \vairtilde \cdot  \nablaH \vairbar \|^2_{\rL^2(\Omegaair)}  
+ \| \vairtilde \cdot \nablaH \vairtilde \|^2_{\rL^2(\Omegaair)} \\
 &\quad + \| \omegaair \vocn_p \|^2_{\rL^2(\Omegaair)} + \| \nablaH \Phi \|^2_{\rL^2(G)}  + \| \Tilde{f}^{\air} \|^2_{\rL^2(\Omegaair)} +\|\nabla \vair_p \|^2_{\rL^2(\Omegaair)} 
+ \| \vair_p \|^2_{\rL^2(\Omegaair)} \bigr)
\end{aligned}
\end{equation*}
as well as
\begin{equation*}
	\begin{aligned}
	&\quad \partial_t \| \nabla \vocn \|^2_{\rL^2(\Omegaocn)} + \| \Delta        \vocn\|^2_{\rL^2(\Omegaocn)} -\int_G \dt \vocn |V|V \d \xH \\
 &\leq C \bigl(\| \vocnbar \cdot \nablaH \vocnbar \|^2_{\rL^2(G)}+ \| \vocnbar \cdot \nablaH \vocntilde \|^2_{\rL^2(\Omegaocn)} +  \| \vocntilde \cdot \nablaH \vocnbar \|^2_{\rL^2(\Omegaocn)}  
+ \| \vocntilde \cdot \nablaH \vocntilde \|^2_{\rL^2(\Omegaocn)} \\
 &\quad + \| w \vocn_z \|^2_{\rL^2(\Omegaocn)} + \| \nablaH \pi_s \|^2_{\rL^2(G)} + \| \Tilde{f}^{\ocn} \|^2_{\rL^2(\Omegaocn)}\bigr),
	\end{aligned}
\end{equation*}
where we used  $-\Delta^{\atm} \vair =-\DeltaH \vair - \partial_p (p \partial_p \vair) =- \Delta \vair + (1-p^2) \partial_{pp} \vair - 2p \partial_p \vair$.
An addition of both equations and a combination of the estimates results in
\begin{equation*}
    \begin{aligned}
        &\quad \partial_t \bigl(\| \nabla \vair \|^2_{\rL^2(\Omegaair)} +\| \nabla \vocn \|^2_{\rL^2(\Omegaocn)} \bigr) + \delta \bigl(  \| \Delta \vair\|^2_{\rL^2(\Omegaair)} + 
\| \Delta \vocn\|^2_{\rL^2(\Omegaocn)} \bigr) + \frac{1}{3} \dt\|V \|^3_{\rL^3(G)} \\
        & \leq K_3(t) \bigl( \left \| \nabla \vair  \right \|^2_{\rL^2(\Omegaair)} + \| \nabla \vocn \|^2_{\rL^2(\Omegaocn)} \bigr) + K_4(t),
    \end{aligned}
\end{equation*}
with $\delta >0$, and $K_3$, $K_4$ are functions depending on $\| v_z\|_{\rL^2(\Omega)}$, $\| \nabla v_z\|_{\rL^2(\Omega)}$, $\| \vbar \|_{\rH^1(G)}$, $\| \vtilde \|_{\rL^4(\Omega)}$, $\| \pi \|_{\rL^2(G)}$ 
and $\| f \|_{\rL^2(\Omega)}$.
The functions $K_3$ and $K_4$ are hence integrable with respect to time by \autoref{aprioristep4}, so Gronwall's inequality yields the existence  of a continuous function $B_2$, depending on the initial data, the forcing terms and time, such that for $t \in [0,T]$, we have
\begin{equation*}
  \| \nabla \vair \|^2_{\rL^2(\Omegaair)} + \| \nabla \vocn \|^2_{\rL^2(\Omegaocn)} +\int_0^t \| \Delta \vair \|^2_{\rL^2(\Omegaair))} \d s+ \int_0^t\| \Delta \vocn \|^2_{\rL^2(\Omegaocn))} \d s  \leq B_2(t). \qedhere
\end{equation*}
\end{step}
\end{proof}

We finally  give a-priori bounds in the maximal regularity space $\E_{1,1,T}$ for $p =q=2$, leading to $\mu_c = 1$.

\begin{prop}[$\E_{1,1}$-Estimate] \label{prop:estimatemaxreg}
Let $(\vair, \vocn)$, $(\vair_0,\vocn_0)$ and $(\fair, \focn)$ be as in \eqref{eq:Assumptionsapriori} and \eqref{eq:Assumptionsapriori2}. 
Then there is a continuous function $B_3$ on $[0,T]$, depending on $\| \vair_0 \|_{\rH^1(\Omegaair)}$, $\| \vocn_0 \|_{\rH^1(\Omegaocn)}$, $\| f^{\air} \|_{\rL^2(0,T;\rL^2(\Omegaair))}$, 
$\|\focn \|_{\rL^2(0,T;\rL^2(\Omegaocn))}$ and $t$, such that 
\begin{equation*}
	\| \vair \|_{\E_{1,1, \air}} + \| \vocn \|_{\E_{1,1,\ocn} } \leq B_3(t), \quad t \in[0,T].
\end{equation*}
\end{prop}
		
\begin{proof}
As before, we will not distinguish between $\vair$ and $\vocn$ in the proof. 
Thanks to \autoref{thm:optdata}, we obtain
\begin{equation*}
        \left \| v \right \|_{\E_{1,1,T}} \leq C \left ( \|v_0\|_{\rH^1(\Omega)} + \| f \|_{\E_{0,1,T}}  + \| G(v) \|_{\E_{0,1,T}} + \| B(v) \|_{\F_{1,T}}  \right)
\end{equation*}
It remains to bound the nonlinear terms  $G(v)= v\cdot \nablaH v + w \dz v$ and $B(V)= \pm V|V|$ by the $\rL^\infty \rH^{1}$- and~$\rL^2 \rH^{2}$-norms.
Concerning $G$, we verify that  $\left \| G(v) \right \|^2_{\E_{0,1,T}} \leq C \left \|v \right \|^2_{\rL^\infty(0,T;\rH^{1}(\Omega))} \cdot  \left \|v \right \|^2_{\rL^2(0,T;\rH^{2}(\Omega))}  \leq CB_2^4$.
Next, observe that $\| B (V) \|_{\F_{1,T}} \leq C \| V|V| \|_{\rL^2(0,T; \rB_{22}^{ \nicefrac{1}{2}}(G))} +C \| V|V| \|_{\rF_{22}^{\nicefrac{1}{4}}(0,T; \rL^2(G))}$.
Following the lines of~\eqref{eq:estimateBbesov}, and choosing $r=r'=2$, we obtain for  $p=q=2$ the estimate
\begin{equation*}
\begin{aligned}
\| V|V| \|_{\rL^2(0,T; \rB_{22}^{ \nicefrac{1}{2}}(G))} 
        &\leq C \| V \|_{\rL^2(0,T;\rW^{\nicefrac{3}{4}+\eps,4}(\Omega))} \cdot \| V \|_{\rL^\infty(0,T;\rW^{\nicefrac{1}{4},4}(\Omega))}  \\ 
  &\leq C  \| V \|_{\rL^2(0,T;\rH^{2}(\Omega))} \cdot \| V \|_{\rL^\infty(0,T;\rH^{1}(\Omega))} 
        \leq C B_2^2.
	\end{aligned}
\end{equation*}
In a similar way, invoking the maximal regularity embedding $\E_{1,1} \hookrightarrow \rH^{\nicefrac{1}{2}}(0,T;\rH^{1}(\Omega))$ and \autoref{lem:weighted paraprod estimate} as well as 
\eqref{eq:rel emb Bessel pot Triebel-Lizorkin}, see also \cite[Proposition~1.2]{Chae:02}, we find that 
\begin{equation*}
	\begin{aligned}
		\| V|V| \|_{\rF_{22}^{\nicefrac{1}{4}}(0,T;\rL^2(G))} 
	    &\leq C \| V \|_{\rF^{\nicefrac{1}{4}}_{22}(0,T; \rL^4(G))} \cdot \| V \|_{\rL^\infty(0,T;\rL^4(G))} \\
        &\leq C \| V \|_{\rH^{\nicefrac{1}{4}}(0,T; \rW^{\nicefrac{1}{4},4}(\Omega))} \cdot \| V \|_{\rL^\infty(0,T;\rW^{\nicefrac{1}{4},4}(\Omega))} \\ 
        &\leq C \| V \|^{\nicefrac{1}{2}}_{\rH^{\nicefrac{1}{2}}(0,T; \rW^{\nicefrac{1}{4},4}(\Omega))}\cdot  \| V \|^{\nicefrac{1}{2}}_{\rL^2(0,T; \rW^{\nicefrac{1}{4},4}(\Omega))} \cdot \| V \|_{\rL^\infty(0,T;\rW^{\nicefrac{1}{4},4}(\Omega))} \\ 
        &\leq\eps \| V \|_{\E_{1,1,T}}+ C \| V \|^3_{\rL^\infty(0,T;\rH^{1}(\Omega))}  
        \leq  \eps \bigl( \|\vair \|_{\E_{1,1,\air}} + \|\vocn \|_{\E_{1,1,\ocn}}\bigr) + C B_2^3.
	\end{aligned}
\end{equation*}
An absorption argument yields the desired estimate with $B_3 := C(B_2^4+ B_2^3 + B_2^2)$. 
\qedhere
\end{proof}

\begin{proof}[Proof of \autoref{thm:globalwellposed}]
Combining \autoref{prop:localwellposed} with  \autoref{thm:blowup} and \autoref{prop:estimatemaxreg}, we conclude that the unique, local, strong solution to \eqref{eq:CAO} exists globally provided $p=q=2$.
We thus obtain the assertion of our main result, \autoref{thm:globalwellposed}, for the particular situation of $p=q=2$. 
 
To treat  the general situation of $p,q$, we consider first $p \geq 2$ and $ q=2$. 
Let $(\vair, \vocn) \in \E_{1,\mu_c,T'}$ be the local, strong  solution of \eqref{eq:CAO} given by \autoref{prop:localwellposed}. 
Note that $\E_{1,\mu_c,T'} \hookrightarrow C((\delta,T'];\rH_{\sigmabar}^1(\Omega^\atm) \times \rH_{\sigmabar}^1(\Omega^\ocn))$. We hence find $T''$ with $\delta < T''<T'$ and taking $v(T'') \in \rH_{\sigmabar}^1(\Omega^\atm) \times \rH_{\sigmabar}^1(\Omega^\ocn)$ as new initial value, we observe that 
$v$ is a solution within the $p=q=2$-framework and exists thus globally. 
The embeddings
\begin{equation*}
    \E_{1,1,T'} \hookrightarrow \rH^{\theta,2}(0,T';\rH^{2(1-\theta),2}(\Omega)) \hookrightarrow \rL^p(0,T';\rH^{1+\nicefrac{2}{p},2}(\Omega)),
\end{equation*}
for $\theta = \frac{1}{2}- \frac{1}{p}$, combined with  \autoref{thm:blowup} yield that $v$ is a global solution for the case $p \geq 2$. 

We now study the case  $2 \leq q \leq p$. Let again $v$ be the local solution. As above, we find $T''$ such that~$\delta < T''<T'$ and consider $v(T'')$  as initial value within the $p,2$-framework. Noting 
that  
\begin{equation*}
    \E_{1,\mu_c,T'} \subset \rL^p_{\mu_c}(0,T';\rH^{2}(\Omega)) \hookrightarrow \rL^p(0,T';\rH^{\nicefrac{2}{p}+\nicefrac{2}{q},q}(\Omega)) = \rL^p(0,T';\rX_{\mu_c})
\end{equation*}
holds true in particular  whenever  $2\leq q \leq p$, we see  by \autoref{thm:blowup} that $v$ exists globally.

Finally, let $q\in (1,2)$. The above argument shows that $v$ exists globally within the $p,2$-framework.  Thanks to
$\rH^2(\Omega) \subset \rH^{2,q}(\Omega)$, $v$ is already a unique, global, strong solution in $\rL^q$, completing the proof. 
\qedhere 
\end{proof}

\section{Higher regularity}\label{sec:regularity}
We start by noting that the assertions of  \autoref{cor:regularity} are based on Angenent's parameter trick, see \cite{A:90,A:90(1)}.
Note that \autoref{cor:regularity}(i) is a property in the interior and is hence not effected by the boundary conditions. It follows thus from the arguments in \cite{GGHHK:20}. 
The same is valid for \autoref{cor:regularitywithout}.
In the sequel,  we will only prove \autoref{cor:regularity}(ii), since part~(iii) can be shown in a similar manner. 

\begin{proof}[Proof of {\autoref{cor:regularity}(ii)}.]
To simplify notation, we only consider $G = \R^2$ and identify $\Omegaatm$ with $G \times (0,\hatm)$. This can be obtained  rigorously using the translation in $p$-coordinate as a chart. 

Let $x_{\mathrm{H},0} = (x_0,y_0,0) \in \Gamma$ be a point at the interface. Further, let $R > 0$ and $3r^{\ocn} \in (0,1)$ as well as~$3r^\atm \in (p_a,p_s)$.
Consider the truncated shifts
   \begin{align*}
       \tau_{\xi}^{\atm} (x_{\mathrm{H}},p) &:= (x_{\mathrm{H}}+ \xi \zeta_0^{\atm}(p),p),
       \quad \for x_{\mathrm{H}} \in B_{\R^2}(0,3R) \quad \text{and} \quad p \in (3r^{\atm},p_s], \\
       \tau_{\xi}^{\ocn} (x_{\mathrm{H}},z) &:= (x_{\mathrm{H}}+ \xi \zeta_0^{\ocn}(z),z),
       \quad \for x_{\mathrm{H}} \in B_{\R^2}(0,3R) \quad \text{and} \quad z \in (-3 r^\ocn, 0], 
   \end{align*}
 where $\zeta_0^{\atm}, \zeta_0^{\ocn}$ are smooth cut-off functions on $\R$ equal to $1$ for $|z| \leq 2 r$ and 0 for $|z| > \frac{5}{2} r$. Here $\xi \in B_{\R^2}(0,3R)$ acts only in tangential direction. 
 Besides, define neighbourhoods $U^{\atm} := B_{\R^2}(0,3R) \times (3r^{\atm},0)$ of $(\xH,p_s)$ and~$U^{\ocn} := B_{\R^2}(0,3R) \times (-3r^{\ocn},0)$ of $(x_{\mathrm{H},0},0)$, and set
    \begin{equation*}
        \tau_{\lambda,\xi}^{\atm}(t,\xH,p) := 
        \begin{cases} 
        (t+\lambda t, \tau_{t\xi}(\xH,p_s)), &\text{ for } (t,\xH,p_s) \in (0,T') \times U^{\atm}, \\
        (t+\lambda t, \xH,p), &\text{ for } (t,\xH,p) \in (0,T') \times \Omegaatm \setminus U^{\atm}
        \end{cases}  
    \end{equation*}
    as well as 
        \begin{equation*}
    	\tau_{\lambda,\xi}^{\ocn}(t,\xH,z) := 
    	\begin{cases} 
    		(t+\lambda t, \tau_{t\xi}(\xH,0)), &\text{ for } (t,\xH,0) \in (0,T') \times U^{\atm}, \\
    		(t+\lambda t, \xH,z), &\text{ for } (t,\xH,z) \in (0,T') \times \Omegaocn \setminus U^{\ocn},
    	\end{cases}  
    \end{equation*}
where $T' \in (0,T)$ is such that $(1+\lambda)t \in (0,T)$ for all $t \in (0,T')$. In addition, define the push-forward operator by 
    \begin{align*}
        T_{\lambda,\xi} \binom{\vatm}{\vocn}
        := \binom{\vatm \circ \tau_{\lambda,\xi}^{\atm}}{\vocn \circ \tau_{\lambda,\xi}^{\ocn}} .
    \end{align*}
Note that $T_{\lambda,\xi} \colon \E_{0,\mu_c,T} \to \E_{0,\mu_c,T}$ is an isomorphism. Moreover, $T_{\lambda,\xi} \colon \E_{1,\mu_c,T} \to \E_{1,\mu_c,T}$ is also an isomorphism, 
since $\tau^{\atm}_{\lambda,\xi}$ and $\tau^{\ocn}_{\lambda,\xi}$ act only tangentially. Furthermore, we set 
    \begin{align*}
    	\tau^{\atm} \colon (-r^\atm,r^\atm) \times B_{\R^2}(0,3R) \to \mathrm{Diff}^\infty((0,T')\times \Omegaatm) \colon (\lambda,\xi) \mapsto \tau_{\lambda,\xi}^{\atm}, \\
    	\tau^{\ocn} \colon (-r^\ocn,r^\ocn) \times B_{\R^2}(0,3R) \to \mathrm{Diff}^\infty((0,T')\times \Omegaocn)\colon (\lambda,\xi) \mapsto \tau_{\lambda,\xi}^{\ocn}.
    \end{align*}
Finally, define $\tilde{\mathrm{H}} \colon \E_{1,\mu_c,T} \to \E_{0,\mu_c,T} \times \F_{\mu_c,T} \times \rB_{qp,\sigmabar}^{\nicefrac{2}{q}}(\Omegaatm) \times  \rB_{qp,\sigmabar}^{\nicefrac{2}{q}}(\Omegaocn)$ such that 
$\tilde{\mathrm{H}}(\vatm,\vocn) = (\fatm,\focn,0,\vatm_0,\vocn_0)$ is equivalent to \eqref{eq:CAOsimplified} and
    \begin{equation*}
    	\mathrm{H}(\lambda,\xi,\vatm,\vocn) := T_{\lambda,\xi}\tilde{\mathrm{H}}(T_{\lambda,\xi}^{-1} (\vatm,\vocn)^\top ) .
    \end{equation*}
Let us note that $\mathrm{H} \in \rC^1$. The arguments given in  \autoref{sec:linearizedproblem} imply then that the Fr{\'e}chet derivative 
$D_{(\vatm,\vocn)} \mathrm{H}(0,0,\widehat\vatm,\widehat\vocn):$ $\E_{1,\mu_c,T} \to \E_{0,\mu_c,T}\times \F_{\mu_c,T} \times \rB_{qp,\sigmabar}^{\nicefrac{2}{q}}(\Omegaatm) \times  \rB_{qp,\sigmabar}^{\nicefrac{2}{q}}(\Omegaocn)$ with respect to 
$(\vatm,\vocn)$ at $(0,0,\widehat\vatm,\widehat\vocn)$ with $\mathrm{H}(0,0,\widehat\vatm,\widehat\vocn) = (\fatm,\focn,0,\vatm_0,\vocn_0)$ 
is again an isomorphism. Hence, by the implicit function theorem, there exists $\delta > 0$ and a function $\Phi \in \rC^1((-\delta,\delta)\times B_{\R^2}(0,\delta); \E_{1,\mu_c,T})$ such that
\begin{equation*}
    \mathrm{H}(\lambda,\xi,\Phi(\lambda,\xi)) = 0
\end{equation*}
for all $\lambda \in (-\delta,\delta)$ and $\xi \in B_{\R^2}(0,\delta)$.
The trace theorem implies 
    \begin{equation*}
    \tr_{\Gamma} \E_{1,\mu_c,T} = \F_{\mu_c,T} = \rF^{1-\nicefrac{1}{2p}}_{p,q,\mu_c}(0,T;\rLq(\Gamma;\R^2)) \cap \rL^p_{\mu_c}(0,T;\rB_{qq}^{2-\nicefrac{1}{q}}(\Gamma;\R^2)),
    \end{equation*} 
and the mixed derivative theorem yields 
    \begin{align*}
        \rF^{1-\nicefrac{1}{2p}}_{p,q,\mu_c}(0,T;\rLq(\Gamma;\R^2)) \cap \rL^p_{\mu_c}(0,T;\rB_{qq}^{2-\nicefrac{1}{q}}(\Gamma;\R^2))
        &\hookrightarrow 
        \rF^{(1-\nicefrac{1}{2p})\cdot\theta}_{p,q_\theta,\mu_c}(0,T;\rB_{qq}^{(2-\nicefrac{1}{q})\cdot (1-\theta)}(\Gamma;\R^2)) \\
        &\hookrightarrow
        \rC^\alpha(0,T;\rC^{1,\alpha}(\Gamma)),
    \end{align*}
for  $\frac{1}{q_\theta} = \frac{\theta}{q} + \frac{(1-\theta)}{2}$ provided $( 1-\frac{1}{2p})\theta - \frac{1}{p} > \alpha$ and $(2-\frac{1}{q})(1-\theta) - \frac{2}{q} > 1+\alpha$. 
A parameter $\theta \in (0,1)$ can be found provided  $3 \alpha + \frac{2}{p} + \frac{3}{q} < 1$.  Hence, 
$\vatm,\vocn, \nablaH \vatm, \nablaH \vocn \in \rC^\alpha((0,T) \times \Gamma)$ for  $\alpha \in (0, \nicefrac{1}{3}-\nicefrac{2}{3p}-\nicefrac{1}{q})$.
For fixed $(t_0,x_{\mathrm{H},0})$, we see that  
    \begin{equation*}
        (\lambda,\xi) \to \binom{\vatm_{\lambda,\xi}}{\vocn_{\lambda,\xi}}(t,x_{\mathrm{H}}) 
        = \binom{\vatm ((1+\lambda)t,x+ t \cdot x_{\mathrm{H}},p_s)}{\vocn ((1+\lambda)t,x+ t \cdot x_{\mathrm{H}},0)}
    \end{equation*}
belongs to  $\rC^1$ and even to  $\rC^{1,\alpha}$ if $\alpha \in (0, \nicefrac{1}{3}-\nicefrac{2}{3p}-\nicefrac{1}{q})$. Thus, we have $\vatm$, $\vocn$, $\nablaH \vatm$, $\nablaH \vocn \in \rC^{1,\alpha}((0,T) \times \Gamma)$
for  $\alpha \in (0, \nicefrac{1}{3}-\nicefrac{2}{3p}-\nicefrac{1}{q})$ near $(t_0,x_{\mathrm{H},0})$.  
\end{proof}

\medskip 

{\bf Acknowledgements}
Tim Binz would like to express his gratitude to DFG for support via project  538212014.
Felix Brandt acknowledges the support by the German National Academy of Sciences Leopoldina through the Leopoldina Fellowship Program with grant number LPDS~2024-07.
Matthias Hieber and Tarek Z\"ochling acknowledge the support by DFG project FOR~5528.

\end{document}